\documentclass[letterpaper, 11pt]{amsart}
\pdfoutput=1
\usepackage{graphicx} 

\usepackage[french,english]{babel}
\usepackage{graphicx}
\usepackage{subcaption}
\usepackage{tabularx}
\usepackage{booktabs}
\usepackage{array}
\usepackage{amsmath}
\usepackage{amsfonts}
\usepackage{amssymb}
\usepackage{amsthm}
\usepackage{bbm}
\usepackage{float}
\usepackage[scr]{rsfso}
\usepackage{enumitem}
\usepackage{permute}
\usepackage{slashed}
\usepackage[usenames,dvipsnames]{xcolor}
\usepackage[pagebackref = false, colorlinks, linkcolor = Red, citecolor = Green, bookmarksdepth=2, linktocpage=true]{hyperref}
\usepackage[capitalize]{cleveref}
\usepackage{verbatim}

\crefformat{footnote}{#2\footnotemark[#1]#3}

\usepackage[T1]{fontenc}

\usepackage{tikz}
\usetikzlibrary{calc}
\usetikzlibrary {arrows.meta}

\newcommand{\op}{\operatorname}

\newcommand{\be}{\begin{equation}}
	\newcommand{\ee}{\end{equation}}
\newcommand{\Ga}{\Gamma}

\newcommand{\R}{\mathbb R}
\renewcommand{\H}{\mathbb H}

\newcommand{\N}{\mathbb N}
\newcommand{\ga}{\gamma}

\newcommand{\La}{\Lambda}
\newcommand{\inte}{\op{int}}
\newcommand{\ba}{\backslash}

\newcommand{\cal}{\mathcal}
\newcommand{\br}{\mathbb R}
\newcommand{\SO}{\op{SO}}

\newcommand{\PSL}{\op{PSL}}
\newcommand{\F}{\cal F}

\newcommand{\vol}{\op{Vol}}

\newcommand{\G}{\Gamma}
\newcommand{\m}{\mathsf{m}}

\newcommand{\T}{\op{T}}
\renewcommand{\frak}{\mathfrak}
\renewcommand{\v}{\mathsf v}
\newcommand{\e}{\varepsilon}

\newcommand{\BMS}{\mathsf{m}_\v^{\op{BMS}}}
\renewcommand{\L}{\cal L}
\newcommand{\fa}{\mathfrak a}

\renewcommand{\i}{\op{i}}

\renewcommand{\S}{\mathbb S}

\newcommand{\PGL}{\op{PGL}}

\renewcommand{\P}{\mathbb{P}}

\renewcommand{\fg}{\frak g}

\newcommand{\rank}{r}

\newcommand{\supp}{\op{supp}}

\newcommand{\bp}{\bs{\varphi}}
\newcommand{\sQ}{\sf{Q}}
\newcommand{\w}{\sf{w}}

\DeclareMathOperator{\interior}{int}

\DeclareMathOperator{\Hom}{Hom}

\DeclareMathOperator{\Ad}{Ad}

\newcommand{\LieG}{\mathfrak{g}}
\newcommand{\LieA}{\mathfrak{a}}

\newcommand{\LieK}{\mathfrak{k}}

\newcommand{\LieP}{\mathfrak{p}}

\newcommand{\Fboundary}{\cal{F}}

\newcommand{\involution}{\mathsf{i}}
\newcommand{\growthindicator}{\psi_\Gamma}

\newcommand{\limitcone}{\cal{L}}
\renewcommand{\u}{\mathsf{u}}

\newcommand{\BR}{\mathsf{m}^{\mathrm{BR}}_\v}
\newcommand{\BRstar}{\mathsf{m}^{\mathrm{BR}_*}_\v}
\renewcommand{\T}{\mathbb T}
\newcommand{\rankG}{\mathrm{rank}(G)}
\newcommand{\primGamma}{\Gamma_{\mathrm{prim}}}

\newcommand{\bs}{\boldsymbol}
\renewcommand{\sf}{\mathsf}
\newcommand{\bv}{\bs{\varphi}}
\newtheorem{theorem}{Theorem}[section]

\newtheorem*{claim*}{Claim}

\newtheorem{lemma}[theorem]{Lemma}

\newtheorem{proposition}[theorem]{Proposition}

\theoremstyle{definition}
\newtheorem{definition}[theorem]{Definition}

\theoremstyle{remark}
\newtheorem{remark}[theorem]{Remark}

\numberwithin{equation}{section}

\title[Multiple correlations of
spectra]{Multiple correlations of 
	spectra \\ for higher rank Anosov representations}

\author{Michael Chow}
\address{Department of Mathematics, Yale University, New Haven, CT 06511}
\email{mikey.chow@yale.edu}

\author{Hee Oh}
\address{Department of Mathematics, Yale University, New Haven, CT 06511 and Korea Institute for Advanced Study, Seoul}
\email{hee.oh@yale.edu}
\thanks{
	Oh is partially supported by the NSF grant DMS-1900101.}

\begin{document}
	\begin{abstract} 
		We describe multiple correlations of  Jordan and Cartan spectra for any finite number of Anosov representations of a finitely generated group.	This extends our previous work on correlations of length and displacement spectra for rank one convex cocompact representations. Examples include correlations of the Hilbert length spectra  for convex projective structures on a closed surface as well as correlations of eigenvalue gaps and singular value gaps for Hitchin representations. We relate the correlation problem  to the counting problem for Jordan and Cartan projections of an Anosov subgroup with respect to a family of carefully chosen truncated {\it hypertubes}, rather than in tubes as in our previous work.  Hypertubes go to infinity in a linear subspace of directions, while tubes go to infinity in a single direction and this feature presents a novel difficulty in this higher rank correlation problem.
	\end{abstract}
	
	\maketitle
	
	\section{Introduction}
	\label{sec:Introduction} 	
	In this paper, we describe multiple correlations of  Jordan and Cartan spectra for any finite number of Anosov representations of a finitely generated group into semisimple real algebraic groups.
	
	We begin by discussing special cases of our result for different classes of geometric structures on a closed surface. Let $S$ be a closed orientable surface of genus $g \ge 2$.
	Let $\rho$ be a discrete faithful representation of the fundamental group $\Sigma=\pi_1(S)$  into $\PSL_2\R$, which can be viewed as an element of the Teichm\"{u}ller space $\cal{T}(S)\simeq \br^{6g-6}$ after identifying conjugate representations. Hence $\rho$ induces a hyperbolic structure $S_\rho \simeq \rho(\Sigma) \backslash\H^2$ and a length function $\ell_\rho$ which assigns to the conjugacy class $[\sigma] \in [\Sigma]$ the hyperbolic length of the corresponding closed geodesic in $S_\rho$. The prime geodesic theorem due to Huber \cite{Hub59} gives an asymptotic for the number of closed geodesics in $S_\rho$ of hyperbolic length at most $T$:
	$$\#\{[\sigma] \in [\Sigma]: \ell_\rho(\sigma) \le T\} \sim \frac{e^T}{T} \quad \text{as } T \to \infty.$$ 
	Given a $d$-tuple $\bs{\rho}=(\rho_1,\dots,\rho_d)$ of distinct elements in $\cal{T}(S)$, it is natural to ask how the length functions $\ell_{\rho_1}, \dots, \ell_{\rho_d}$ are correlated. 
	In \cite{CO23}, we proved that for any interior vector $\bs{r}=(r_1, \dots, r_d)$ in the smallest closed cone containing all vectors $(\ell_{\rho_1}(\sigma),\dots,\ell_{\rho_d}(\sigma))$ for $\sigma \in \Sigma$, there exists $\delta=\delta(\bs{\rho}, \bs{r})
	> 0$ such that for any $\e_1,\dots,\e_d > 0$, we have as $T \to \infty$,
	$$\#\{[\sigma] \in [\Sigma]: r_iT \le \ell_{\rho_i}(\sigma) \le r_iT + \e_i \text{ for all $1\le i\le d$}\} \sim c\cdot\frac{e^{\delta T}}{T^{(d+1)/2}}$$ 
	for some $c > 0$. This was earlier proved by Schwarz-Sharp \cite{SS93} for $d=2$ and  $\bs{r}=(1,1)$ (see \cite{Glo17} for $d=2$ and a general $\bs{r}$).
	Indeed, we obtained an analogous result on correlations of length spectra for any finite number of  rank one convex cocompact representations of a finitely generated group \cite{CO23}. 
	
	\subsection*{Hilbert length spectra for convex projective structures}
	Another space of interesting geometric structures on a closed surface $S$ is the space $\cal{C}(S)$ of convex projective structures. A discrete faithful representation $\rho:\Sigma \to \PGL_3\R$ is called {\it{convex projective}} if its image acts cocompactly on some properly convex domain $\Omega \subset \R\P^2$. Such a representation $\rho$ endows a convex projective structure on $S$
	which we denote by $S_\rho\simeq \rho(\Ga)\ba \Omega$. The space $\cal{C}(S)$ can be identified with the space of convex projective representations $\rho: \Sigma \to \PGL_3\R$ modulo conjugation. 
	Goldman showed that $\cal{C}(S)$ is homeomorphic to $\R^{16g-16}$ and contains the Teichm\"uller space $\cal{T}(S)$ as a $(6g-6)$-dimensional subspace \cite{Gol90}. 
	For $\rho\in {\cal C}(S)$ and a conjugacy class $[\sigma]\in [\Sigma]$, denote by $\ell^{\op{H}}_\rho(\sigma)$ the Hilbert length of the corresponding closed geodesic (cf. \eqref{hilbert}). Benoist proved the prime geodesic theorem for $S_\rho$:
	\begin{equation}
		\label{eqn:HilbertPrime}
		\#\{[\sigma] \in [\Sigma] : \ell^{\op{H}}_\rho(\sigma) \le T\} \sim \frac{e^{\delta_\rho T}}{\delta_\rho T} \quad \text{as } T \to \infty,
	\end{equation}
	where $\delta_\rho$ is the topological entropy of the Hilbert geodesic flow on $S_\rho$ \cite{Ben04}. 
	
	Blayac and Zhu showed that \eqref{eqn:HilbertPrime} holds more generally for strongly convex cocompact projective representations \cite{BZ23}. Let $\Gamma$ be a finitely generated group. Following \cite[Definition 1.1]{DGK17}, a discrete faithful representation $\rho : \Gamma \to \PGL_3\R$ is {\textit{strongly convex cocompact}} if $\rho(\Gamma)$ acts on some {\textit{strictly}} convex domain $\Omega$, and the action is cocompact on the convex hull of its limit set. 
	For any $d$-tuple $\bs{\rho} = (\rho_1,\dots,\rho_d) : \Gamma \to \PGL_3\R$ of Zariski dense strongly convex cocompact projective representations,
	let $\limitcone_{\bs{\rho}}\subset \R^d$ denote the smallest closed cone containing all vectors $(\ell^{\op{H}}_{\rho_1}(\gamma), \dots, \ell^{\op{H}}_{\rho_d}(\gamma))$ for $\gamma \in \Gamma$. 
	If $\bs{\rho}$ is  {\textit{independent}}, i.e., for all $i \ne j$, $\rho_i$ is conjugate to neither $\rho_j$ nor the contragradient $\rho_j^*$, then $\limitcone_{\bs{\rho}}$ has non-empty interior.
	The following is a special case of our main theorem:
	
	\begin{theorem}[Multiple correlations of Hilbert length spectra]
		\label{thm:Hilbert} Let $d\ge 2$ and
		$\bs{\rho} = (\rho_1,\dots,\rho_d):\Gamma\to \PGL_3(\R)$ be an independent $d$-tuple of Zariski dense strongly convex cocompact projective representations. Then for any interior vector $\bs{r} =(r_1,\dots, r_d) \in \limitcone_{\bs{\rho}}$, there exists $\delta=\delta_{\bs{\rho}}(\bs{r}) > 0$  such that for any $\e_1,\dots, \e_d > 0$, we have as $T \to \infty$,
		$$\#\{[\sigma] \in [\Gamma] : r_iT \le \ell^{\op{H}}_{\rho_i}(\sigma) \le r_iT + \e_i \text{ for all } 1 \le i \le d\} \sim c\cdot \frac{e^{\delta T}}{T^{(d+1)/2}}$$
		for some $c>0$. 
		Moreover, we have the upper bound: $$\delta_{\bs{\rho}} (\bs{r}) \le \min_{1 \le i \le d} \delta_{\rho_i} r_i,$$
		which is strict when all $\delta_{\rho_i}r_i$ are equal. 
	\end{theorem}
	
	\begin{remark}
		Using the thermodynamic approach of Schwartz-Sharp \cite{SS93}, Dai-Martone \cite{DM22} proved \cref{thm:Hilbert} for $d=2$ and $\bs{r}=(1/\delta_{\rho_1},1/\delta_{\rho_2})$. 
	\end{remark}

	\subsection*{Eigenvalue and singular value gaps for Hitchin representations}
	For $n\ge 2$, the Hitchin component $\cal{H}_n$ of $\Hom(\Sigma,\PSL_n\R)$ is the connected component containing the representation $\tau_n \circ \rho_0:\Sigma \to \PSL_n\R$ where $\rho_0: \Sigma \to \PSL_2\R$ is a discrete and faithful representation and $\tau_n:\PSL_2\R \to \PSL_n\R$ is the irreducible representation which is unique up to conjugation \cite{Hit92}. Let $\rho\in \cal H_n$ be a Hitchin representation. Labourie \cite{Lab06} proved that for all nontrivial $\sigma \in \Sigma$, $\rho(\sigma)$ has all distinct positive eigenvalues.
	We denote their logarithms in the decreasing order by
	\begin{equation}
		\label{1}
		\lambda(\rho(\sigma)) = (\lambda_1(\rho(\sigma)), \dots, \lambda_n(\rho(\sigma)))\in \br^n.
	\end{equation} 
	We  denote the logarithms of the singular values of $\rho(\sigma)$ in the decreasing order by
	\begin{equation}
		\label{2}
		\mu(\rho(\sigma)) = (\mu_1(\rho(\sigma)), \dots, \mu_n(\rho(\sigma)))\in \br^n.
	\end{equation}
	For $1 \le k \le n-1$, consider the linear form $\alpha_k:\R^n \to \R$ given by 
	$$\alpha_k(x_1,\dots,x_n) = x_k - x_{k+1} .$$
	Sambarino \cite{Sam14a} proved that
	for any Zariski dense $\rho \in \cal{H}_n$ and $1\le k\le n-1$:
	\begin{equation}\label{sss}
		\#\{[\sigma] \in [\Sigma] : \alpha_k(\lambda(\rho(\sigma))) \le T\} \sim \frac{e^{\delta_{k}T}}{\delta_{k}T} \quad \text{as } T \to \infty.
	\end{equation}
	where $\delta_{k} = \lim_{T \to \infty}\frac{1}{T}\log \#\{[\sigma] \in [\Sigma]: \alpha_k(\lambda(\rho(\sigma)) \le T\}.$  Moreover, Potrie-Sambarino \cite[Theorem B]{PS17} proved that $\delta_k=1$.
	
	For a $d$-tuple $\bs{\rho} = (\rho_1,\dots,\rho_d)$ of Zariski dense Hitchin representations in $\cal{H}_n$ and $\bs{\beta} = (\beta_1,\dots,\beta_d) \in \Pi^d$, where $\Pi=\{\alpha_1,\dots,\alpha_d\}$, denote by $\limitcone_{\bs{\beta}} = \limitcone_{\bs{\rho},\bs{\beta}} \subset \R^d$  the smallest closed cone containing all vectors $$(\beta_{1}(\lambda(\rho_1(\sigma))), \dots, \beta_{d}(\lambda(\rho_d(\sigma)))) \text{ for } \sigma \in \Sigma.$$  
	For example, all $\beta_i$, $1\le i\le d$, can be $\alpha_1$.
	If  $\bs{\rho}$ is {\textit{independent}}, in the sense that  for all $i \ne j$, $\rho_i$ is not conjugate to neither $\rho_j$ nor the contragradient $\rho_j^*$, then $\limitcone_{\bs{\beta}} $ has non-empty interior. We prove the following:

	\begin{theorem}[Multiple correlations of eigenvalue and singular value gaps]
		\label{thm:Hitchin} 
		Let $d\ge 2$ and $\bs{\rho} = (\rho_1,\dots,\rho_d) :\Sigma \to  \PSL_n(\R) $ be a $d$-tuple of independent Zariski dense Hitchin representations. Let $\bs{\beta} = (\beta_1,\dots,\beta_d) \in \Pi^d$. Then for any interior vector $\bs{r} =(r_1,\dots, r_d) \in \limitcone_{\bs{\beta}}$, there exists $\delta=\delta_{\bs{\rho},\bs{\beta}}(\bs{r}) > 0$
		such that for any $\e_1,\dots, \e_d > 0$, we have as $T \to \infty$,
		$$	\#\{[\sigma] \in [\Sigma] : r_iT \le \beta_i(\lambda(\rho_i(\sigma))) \le r_iT + \e_i \text{ for all } 1 \le i \le d\} \sim c \cdot \frac{e^{\delta  T}}{T^{(d+1)/2}} $$
		and
		$$\#\{\sigma \in \Sigma : r_iT \le \beta_i(\mu(\rho_i(\sigma))) \le r_i T + \e_i \text{ for all } 1 \le i \le d\} \sim c' \cdot c \cdot \frac{e^{\delta  T}}{T^{(d-1)/2}} $$
		where $c = c(\bs{\rho},\bs{\beta},\bs{r},\e_1,\dots,\e_d) > 0$ and  $c' = c'(\bs{\rho},\bs{\beta,\bs{r}}) >0$.	Moreover, we have the upper bound \be\label{bdd} \delta_{\bs{\rho},\bs{\beta}}(\bs{r}) \le \min_{1 \le i \le d} r_i\ee  and if all $r_i$ are equal, then the inequality is strict.
	\end{theorem}
	We note that the upper bound \eqref{bdd} is independent of $\bs{\rho}$ and $\bs{\beta}$; the reason behind this is the aforementioned fact that $\delta_k=1$ in \eqref{sss} for all Hitchin representations \cite[Theorem B]{PS17}.

	\begin{remark}  When $d=2$, $\beta_1 = \beta_2$ and $\bs{r}=(1,1)$,
		the eigenvalue gap statement was proved by
		Dai-Martone \cite{DM22}.
	\end{remark}

	\subsection*{Multiple correlations of Jordan and Cartan spectra}
	The aforementioned examples are all {\textit{Anosov} representations} (\cite{FG06}, \cite{Lab06}). The main results of this paper are multiple correlations of Jordan and Cartan spectra for general Zariski dense Anosov representations. Let $G$ be a connected semisimple real algebraic group and fix a Cartan decomposition $G=K (\exp \fa^+) K$ where  $K$ is a maximal compact subgroup, $A$ is a maximal real split torus and
	$\fa^+$ is a positive Weyl chamber of $\fa=\op{Lie}A$.
	Let $\mu:G\to \fa^+$ denote the Cartan projection map, that is, $\mu(g)$ is the unique element of $\fa^+$ such that
	$g\in K (\exp \mu(g)) K$ for all $g\in G$. Let $M$ denote the centralizer of $A$ in $K$.
	
	A finitely generated subgroup $\Gamma < G$ is called Anosov
    (or Borel-Anosov)
	if there exists $C>0$ such that for {\it every} simple root $\alpha$ of $(\fg, \fa)$ and $\gamma\in \Gamma$, we have
	\be\label{al} \alpha(\mu(\gamma)) \ge  C|\gamma| -C^{-1}\ee 
	where $| \cdot |$ denotes the word length of $\Gamma$ with respect to a fixed finite set of generators.\footnote{There is a more general definition of an Anosov subgroup where \eqref{al} is required only for a fixed subset of simple roots, which is not dealt with in this paper.}
	This definition is due to Kapovich-Leeb-Porti \cite{KLP17} (see also \cite{Lab06}, \cite{GW12}, \cite{GGKW17} for other equivalent definitions). There are more general definitions of Anosov subgroups with respect to any non-empty subset of simple roots of $(\frak g, \frak a)$; in this paper, we only concern (Borel-)Anosov subgroups defined as above.
Every nontrivial element $\gamma \in \Gamma$ is loxodromic \cite[Lemma 3.1]{GW12} and hence conjugate to an element $\exp(\lambda(\gamma))m(\gamma) $ where $\lambda(\gamma) \in \interior\LieA^+$ is the Jordan projection of $\gamma$ and $m(\gamma)\in M$.
	The conjugacy class $[m(\gamma)]\in [M]$ is uniquely determined and called the holonomy of $\gamma$. We note that $\lambda(\gamma)$ and $[m(\gamma)]$ depend only on the conjugacy class of $\gamma$. 
	The limit cone $\L_\Ga\subset \fa^+$ of $\Ga$ is the smallest closed cone containing $\lambda(\Ga)$; this is a convex cone with non-empty interior when $\Gamma$ is Zariski dense in $G$ \cite{Ben97}. 
	
	For a Zariski dense Anosov subgroup $\Gamma<G$, Sambarino \cite{Sam14a} proved that  for any $\varphi\in \fa^*$  which is positive on $\limitcone_\Ga - \{0\}$,
	\begin{equation}
		\label{eqn:Sambarino}
		\#\{[\gamma] \in [\Gamma] : \varphi(\lambda(\gamma)) \le T\} \sim \frac{e^{\delta_{\varphi}T}}{\delta_{\varphi}T} \quad \text{as } T \to \infty
	\end{equation}
	where $\delta_{\varphi}=\delta_{\Gamma,\varphi} = \lim_{T \to \infty}\frac{1}{T}\log \#\{[\gamma] \in [\Gamma]: \varphi(\lambda(\gamma)) \le T\}.$
	See also \cite[Corollaries 1.4 and 1.10]{CF23} for a different proof which
	also  includes holonomies.
	
	In order to state our correlation theorems, let $G_1, \dots, G_d$ be any connected semisimple real algebraic groups for $d\ge 2$. Fix a finitely generated group $\Sigma$ and let  
	$$\bs{\rho}=(\rho_i:\Sigma \to G_i)_{1\le i \le d}$$ 
	be a $d$-tuple of Anosov representations\footnote{I.e., $\rho_i(\Sigma)$ is an Anosov subgroup of $G_i$.}. We assume that $\bs{\rho}(\Sigma)$ is Zariski dense in $\prod_{i=1}^d G_i$. 
	We use the same notations for $G_i$ as we did for $G$ but with a subscript $i$. 
	Let  $$\bp:\bigoplus_{i=1}^d \fa_i\to \br^d$$ be a linear map given by
	$\bp=(\varphi_i:\LieA_i \to \R)_{1 \le i \le d}$ where
	$\varphi_i$ is a linear form on $\fa_i$ which is positive on $\limitcone_{\rho_i(\Sigma)} -\{0\}$ for each $i$. Let $\limitcone_{\bs{\varphi}} = \limitcone_{\bs{\rho},\bs{\varphi}} \subset \R^d$ be the smallest closed cone containing all vectors 
	$$\bs{\varphi}(\lambda(\bs{\rho}(\sigma)))=(\varphi_1(\lambda(\rho_1(\sigma))),\dots,(\varphi_d(\lambda(\rho_d(\sigma))) \text{ for } \sigma \in \Sigma.$$ 
	We define the holonomy group $M_{\bs{\rho}}$ of $\bs{\rho}$ as the smallest closed subgroup of $\prod_{i=1}^d M_i$ generated by all the $d$-tuples of holonomies 
	$${ m}(\bs{\rho}(\sigma))=(m(\rho_1(\sigma)),\dots,m(\rho_d(\sigma)) \text{ for } \sigma \in \Sigma.$$ 
	By \cite[Corollary 1.10]{GR07}, this is a finite index normal subgroup of $\prod_{i=1}^d M_i$. Let $\vol_{M_{\bs{\rho}}}$ denote the Haar probability measure on $M_{\bs{\rho}}$. 
	
	The following is the main theorem of this paper:
	\begin{theorem}[Multiple correlations of Jordan and Cartan spectra]
		\label{thm:Correlations} 
		For any interior vector $\bs{r}=(r_1,\dots,r_d) \in \limitcone_{\bs{\varphi}}$, there exists $\delta=\delta_{\bs{\rho},\bs{\varphi}}(\bs{r}) > 0$ such that for any $\e_1,\dots,\e_d > 0$ and for any conjugation invariant Borel subset $\Theta < M_{\bs{\rho}}$ with $\vol_{M_{\bs{\rho}}}(\partial\Theta)=0$, we have as $T \to \infty$, 
		\begin{multline*}
			\#\{[\sigma] \in [\Sigma] : \, \bs{\varphi}(\lambda(\bs{\rho}(\sigma))) \in \prod_{i=1}^d [r_iT, r_iT+\e_i], \, m(\bs{\rho}(\sigma)) \in \Theta\} 
			\\ \sim c\cdot\frac{e^{\delta T}}{T^{(d+1)/2}}\vol_{M_{\bs{\rho}}}(\Theta)
		\end{multline*}
		and
		$$\#\{\sigma \in \Sigma : \bs{\varphi}(\mu(\bs{\rho}(\sigma))) \in \prod_{i=1}^d [r_iT, r_iT+\e_i]\} \sim c' \cdot c\cdot\frac{e^{\delta T}}{T^{(d-1)/2}}$$ 
		where $c = c(\bs{\rho}, \bs{\varphi}, \bs{r}, \e_1,\dots,\e_d) > 0$ and $c' = c'(\bs{\rho}, \bs{\varphi}, \bs{r})>0$.
		Moreover, we have the upper bound
		\be\label{bbb} \delta_{\bs{\rho},\bs{\varphi}}(\bs{r}) \le \min_{1 \le i \le d} \delta_{\rho_i(\Sigma),\varphi_i} r_i\ee 
		which is strict when all $\delta_{\rho_i(\Sigma),\varphi_i} r_i$ are equal.
	\end{theorem}
	
	\begin{remark}
	\begin{itemize}
	    \item 
		The exponent $\delta=\delta_{\bs{\rho},\bs{\varphi}}(\bs{r}) $ is given by $\psi_{\bs{\rho}(\Sigma)}(\v^{\star})$ where $\psi_{\bs{\rho}(\Sigma)}$ is
		the growth indicator of the subgroup $\bs{\rho}(\Sigma)<\prod_{i=1}^d G_i$ and $\v^{\star}=\v^{\star}(\bp, \bs{r})$ is the unique $(\bp,\bs{r})$-critical vector in the interior of the limit cone of $\bs{\rho}(\Sigma)$ (see \cref{thm:Counting} and the proof of \cref{thm:Correlations}). The bound \eqref{bbb} is then deduced from \cite[Theorem 1.4 and its proof]{KMO21} in which it is determined when
		the equality in \eqref{bbb}  happens.
        \item 
        For $\bs{\rho}$ fixed, the exponent $\delta_{\bs{\rho}, \bs{\varphi}}(\bs{r})$ is continuous on the parameter $(\bs{\varphi}, \bs{r})$ because the map 
        $(\bs{\varphi}, \bs{r}) \mapsto \v^{\star} $ 
        is continuous by Lemma
        \ref{lem:Bijection} and the growth indicator 
        $\psi_{\bs{\rho}(\Sigma)}$ is continuous  in the interior of the limit cone of $\bs{\rho}(\Sigma)$ by Quint \cite{Qui02a}. 
        \item For a sequence $\bs{\rho}_k\to \bs{\rho}$ in the space $\text{Hom}(\Sigma,\prod_{i=1}^d G_i)$,
        we have $\delta_{\bs{\rho}_k, \bs{\varphi}}(\bs{r})\to \delta_{\bs{\rho}, \bs{\varphi}}(\bs{r}) $ as $k\to \infty$. This follows from the recent work \cite[Theorem 1.4]{DO}, which says that the growth indicators vary continuously on the deformation space of an Anosov subgroup.
         \end{itemize}
	\end{remark}
	\begin{figure}[H]
		\begin{minipage}{0.45\textwidth}
			\includegraphics[scale = 0.23]{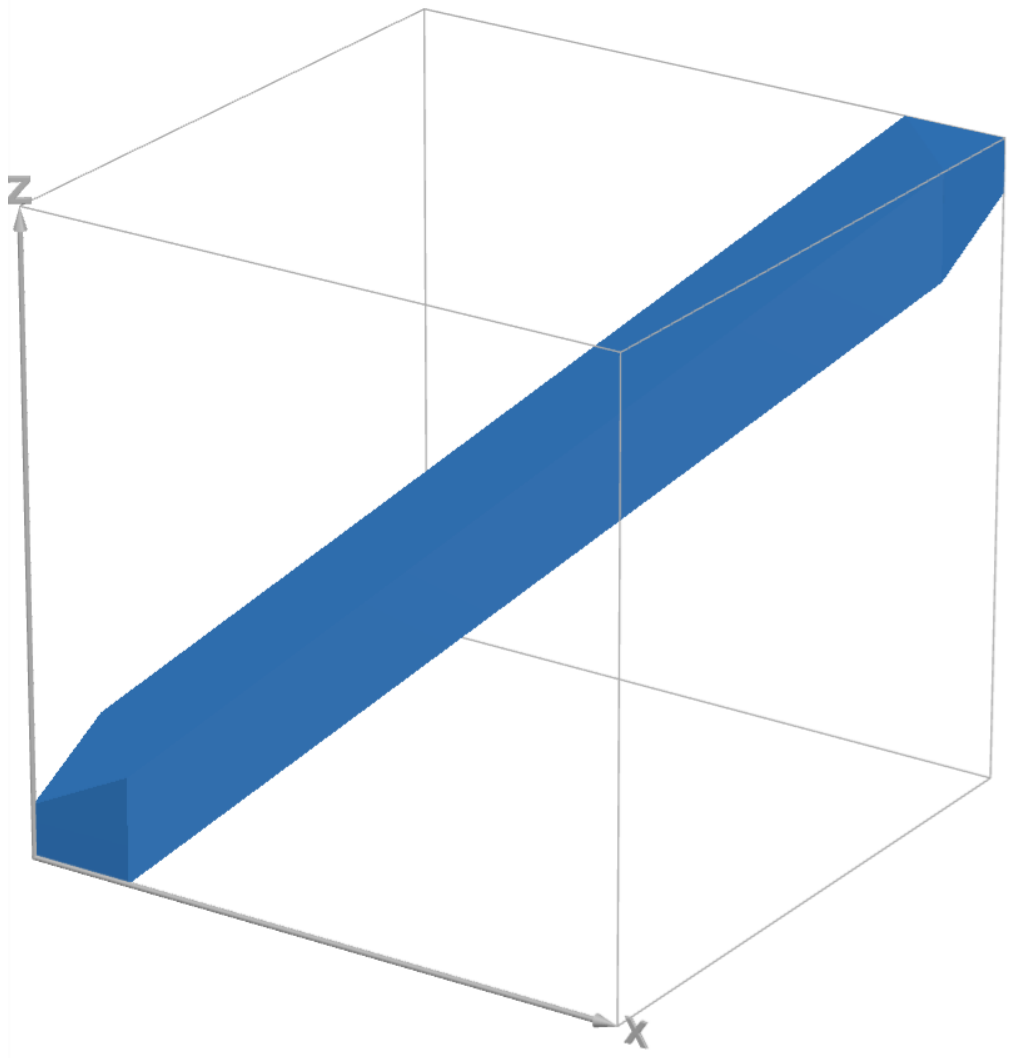}
		\end{minipage}
		\begin{minipage}{0.45\textwidth}
			\includegraphics[scale = 0.23]{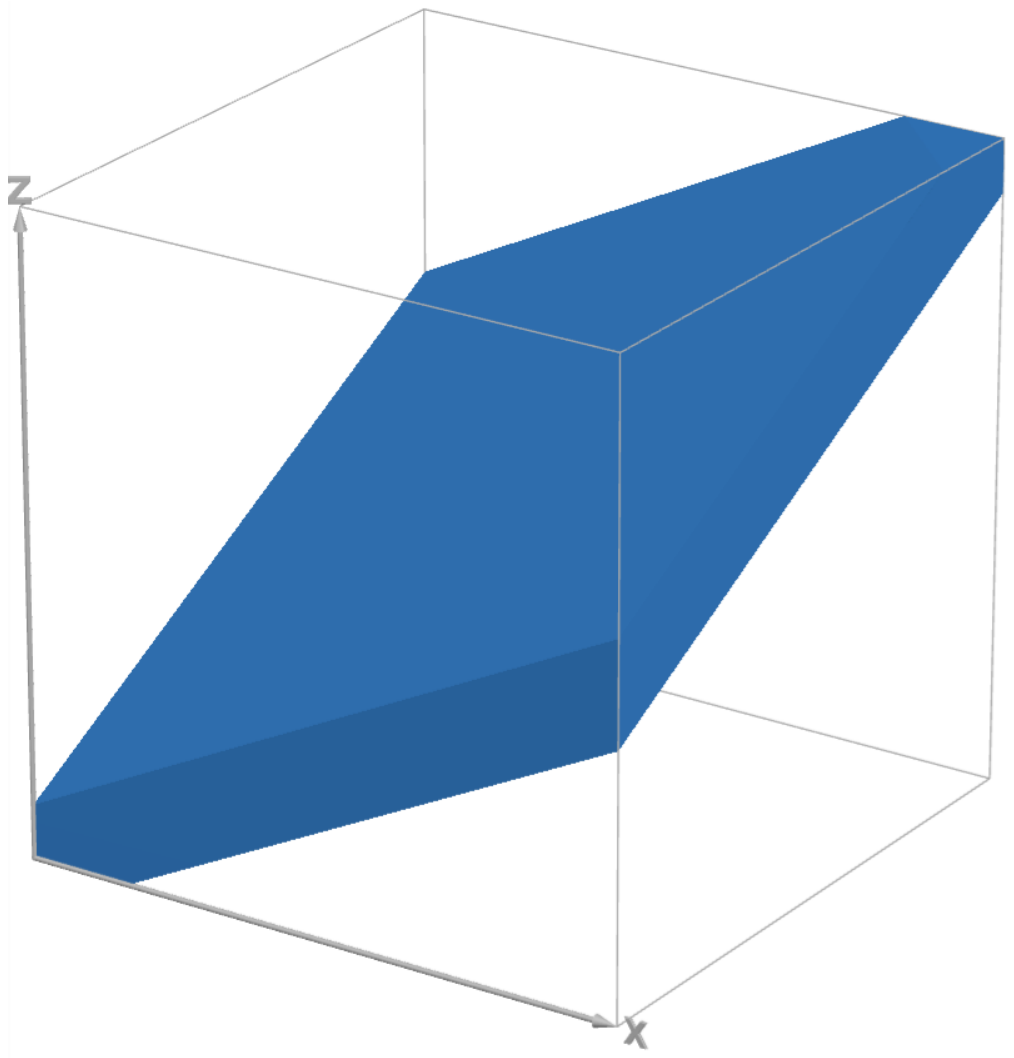}
		\end{minipage}
		\caption{Tube vs. hypertube.} 
		\label{fig:Hypertube}
	\end{figure}

	\subsection*{Comparison with the previous work}  We have previously proved \cref{thm:Correlations} in the case where every $G_i$ has rank one \cite{CO23} (under a slightly stronger assumption on $\bp$). There is a main conceptual difference between that case and the current general case we are dealing with. In that setting, the rank of the product $\prod_{i=1}^d G_i $ is equal to $d$ and hence
	the map $\bp:\bigoplus_{i=1}^d \fa_i\simeq\br^d\to \br^d$ is an isomorphism. Therefore
	the set  $\bs{\varphi}^{-1}(\prod_{i=1}^d [r_iT, r_iT+\e_i]) = \bs{\varphi}^{-1}(\prod_{i=1}^d [0, \e_i]) + T\bs{\varphi}^{-1}(\bs{r})$ is a translation of a compact subset and sweeps out  a {\it tube} going to infinity along the direction of a vector $\bs{\varphi}^{-1}(\bs{r})$ as $T$ increases. Note that there is only {\it one} direction of a tube which tends to infinity and hence only one way to truncate the tubes for studying asymptotics.  
	
	However, when some $G_i$ has rank at least 2, the linear map $$\bp:\bigoplus_{i=1}^d \fa_i \simeq \br^{\sum_{i=1}^d \op{rank} G_i} \to \br^d$$ is not injective and hence the preimage $\bs{\varphi}^{-1}(\bs{r})$ is an affine subspace of positive dimension. As a result, the union $\bigcup_{T>0} \bs{\varphi}^{-1}(\prod_{i=1}^d [r_iT, r_iT+\e_i])$ is  a hypertube which goes to infinity in all directions of the affine subspace $\bs{\varphi}^{-1}(\bs{r})$ (see \cref{fig:Hypertube}). 
	This feature presents a novel difficulty in this higher rank correlation problem, since unlike for tubes, there are many different truncations of the hypertube that can be considered.
	Coming up with a truncation simultaneously adapted to $\bs{\rho}$ and the linear map $\bp$ that can be analyzed is the main challenge in the current setting.

	\subsection*{Outline of the proof of \cref{thm:Correlations}} For simplicity, we will discuss only correlations of Jordan spectra. 	We first explain the relation between correlations of Jordan spectra for a $d$-tuple $\bs{\rho}$ of Anosov representations and counting the Jordan spectra of a single Anosov subgroup $\Gamma$. 
	
	The hypotheses in \cref{thm:Correlations} imply that $\Gamma=\bs{\rho}(\Sigma)$ is a Zariski dense Anosov subgroup of the product $G=\prod_{i=1}^d G_i$, that the cone $\limitcone_{\bs{\varphi}}$ coincides with the image $\bs{\varphi}(\limitcone_{\Ga})$ of the limit cone and that the restriction $\bs{\varphi}|_{\limitcone_{\Ga}}$ is a proper map. We will in fact study a general Zariski dense Anosov subgroup $\Gamma$ of any semisimple real algebraic group $G$ and a general surjective linear map $$\bs{\varphi}:\LieA \to \R^d \;  \text{such that } \bs{\varphi}|_{\limitcone_\Ga}  \text{ is proper}.$$
	In particular, $1\le d \le \rankG = \dim \LieA$ and $\ker\bp$  has dimension $\rankG-d$. Fix a vector $\bs{r}=(r_1, \dots, r_d)$ in the interior of the cone $\bs{\varphi}(\limitcone_\Ga)$ and 
	$\e_1,\dots,\e_d > 0$.
	The properness of $\bs{\varphi}|_{\limitcone_\Ga}$ implies that for each $T>0$,
	\be\label{lll} \#\{[\gamma] \in [\Gamma] : \, \bv(\lambda(\gamma) )\in \prod_{i=1}^d[r_iT,  r_iT+\e_i]\}<\infty
	\ee 
	and
	our goal is to find an asymptotic as $T\to \infty$. Geometrically, for any choice of  $\v\in \fa^+$ with
	$\bp(\v)=\bs{r}$, the shape of the set
	$$\{\u  \in \LieA: \bs{\varphi}(\u) \in \prod_{i=1}^{d}[r_iT,r_iT+\e_i]\} \cap \L_\Ga$$
	is an {\it infinite prism} which slides in the direction of $\v$ as $T$ increases, intersected with the limit cone $\L_\Ga$. 
	We consider their union:
	\begin{equation}
		\label{eqn:Intro1}
		\T  = \bigcup_{T>0}
		\{\u  \in \L_\Ga: \bs{\varphi}(\u) \in \prod_{i=1}^{d}[r_iT,r_iT+\e_i]\}.
	\end{equation} 
	If $\sQ \subset \LieA$ is a compact subset such that $\bs{\varphi}(\sQ) = \prod_{i=1}^{d}[0,\e_i]$ and $V = \bs{\varphi}^{-1}(\R\bs{r})$ is the preimage of the line $\R\bs{r}$, then
	we have (\cref{lem:Hypertube})
	$$\T = (\sQ + V) \cap \L_\Ga;$$
	and a set of this form will be called a hypertube as one can think of $\T$ as obtained from translating $\sQ$ in all directions of $V \cap \L_\Ga$ (see \cref{fig:Hypertube2}).
	
	\begin{figure}[H]
		\includegraphics[scale = 0.27]{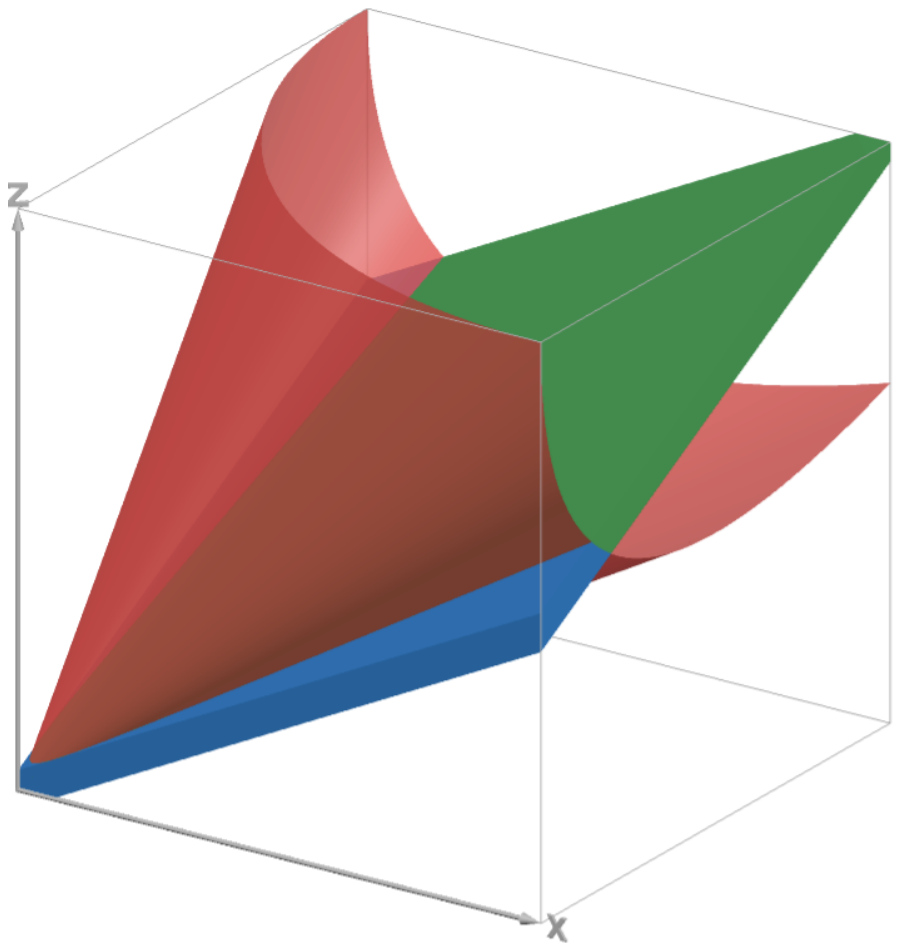}
		\caption{A hypertube of $\limitcone$.} \label{fig:Hypertube2}
	\end{figure}
	
	For each $T>0$, consider a truncation $\T_T:$ \begin{align*}
		\T_T:=\{\u  \in \L_\Ga: \bs{\varphi}(\u) \in \prod_{i=1}^{d}[0,r_iT+\e_i]\}.
	\end{align*}
	
	We show that
	for some vector $\v^{\star} \in \fa^+$, 
	\be\label{asym} \#\{[\gamma] \in [\Gamma]: \lambda(\gamma) \in \T_{T}\} \sim c \cdot \frac{e^{\growthindicator(\v^{\star}) T}}{T^{(d+1)/2}}\quad \text{as $T\to \infty$}\ee 
	where  $\psi_\Ga$ is the growth indicator of $\G$ (see \eqref{gr}) and $c>0$ is a constant (Theorem \ref{thm:JordanCounting}).
	A priori, the existence of $\v^{\star}\in \fa^+$ satisfying \eqref{asym} is not clear at all. To explain where to find such a vector $\v^{\star}\in \fa^+$, observe that  for any vector $\v\in \LieA^+\cap \bp^{-1}(\bs{r})$, the truncation $\T_T$ can be expressed as
	\begin{align*} \T_T = (\sQ + [0,T]\v + \ker\bs{\varphi})\cap\L_\Ga 	.\end{align*}	 When $\T$ is a tube, such a vector $\v$ is uniquely determined, and this $\v$ is $\v^{\star}$, satisfying \eqref{asym}. In a general hypertube case, there is a positive dimensional choice of such $\v\in \bp^{-1}(\bs{r})$. Using the Anosov property of $\Gamma$, we show that there exists a unique vector  $\v^{\star} \in \interior\limitcone\cap  \bp^{-1}(\bs{r})  $ such that 
	$$\ker\bs{\varphi} = V \cap \ker\psi_{\v^{\star}},$$
	where $\psi_{\v^{\star}}\in \fa^*$ is the unique linear form tangent to $\psi_\Ga$ at ${\v^{\star}}$  (\cref{lem:CriticalVector}). We call $\v^{\star}$ a $(\bp,\bs{r})$-critical vector.
	Moreover $\v^{\star}$ satisfies (\cref{lem:sup}):
	$$\psi_\Ga(\v^{\star})= \sup_{\w  \in \interior\limitcone, \, \bs{\varphi}(\w ) = \bs{r}}\psi_{\Gamma}(\w ).$$ 
	
	Once we express the truncation $\T_T$
	as $$\T_T=  (\sQ + [0,T]\v^{\star} + (V\cap \ker\psi_{\v^{\star}}))\cap \L_\Ga   $$ using the critical vector $\v^{\star}$,
	we can	use the framework of our previous work \cite{CO23} to prove \eqref{asym} using local mixing \cref{thm:DecayofMatrixCoefficients}.
	
	The reason  why a formula like \eqref{asym} is relevant to \eqref{lll} is because we can find continuous functions $b_1,b_2 \in C(\sQ )$ such that for all $T > 0$,
	$$\{\u  \in \L_\Ga : \bs{\varphi}(\u) \in \prod_{i=1}^{d}[r_iT,r_iT+\e_i]\} = \T_{T,b_1}(\v^{\star}) - \T_{T,b_2}(\v^{\star})$$
	where 
	$$\T_{T,b_i}(\v^{\star}) = \{\sf{q}  + \v' + t\v^{\star}  \in \T\cap \L_\Ga: \sf{q}  \in \sQ , \, \sf{v}' \in V\cap \ker\psi_{\v^{\star}}, \, 0 \le t \le T +b_i({\sf{q}})\}.$$  
	Indeed, we show \eqref{asym} for
	$\T_{T,b_i}(\v^{\star}) $ in place of $\T_T$ and deduce the asymptotic for \eqref{lll}.

	\subsection*{Organization}
	
	\begin{itemize}
		\item  In \cref{sec:Preliminaries}, we recall some preliminaries on Lie theory and discrete subgroups.
		
		\item  In \cref{sec:CriticalVector}, we specialize to Zariski dense Anosov subgroups $\Gamma < G$ and study linear maps $\bs{\varphi}:\LieA\to \R^d$ which are proper on $\limitcone_\Ga$.  We prove the existence and uniqueness of a $(\bs{\varphi},\bs{r})$-critical direction of $\Gamma$ for any $\bs{r} \in \interior\bs{\varphi}(\limitcone_\Ga)$ (\cref{lem:CriticalVector}).
		
		\item In \cref{sec:Hypertubes}, we define {\textit{hypertubes}} (\cref{def:Hypertube}) and their truncations. We explain the relationship between correlations of linearized Jordan and Cartan spectra for Anosov representations and counting Jordan and Cartan projections in hypertubes of an Anosov subgroup.
		
		\item In \cref{sec:EquidistributionStatements}, we prove joint equidistribution of Jordan projections in hypertubes and holonomies of $\Gamma$ (\cref{thm:JointEquidistribution}) and equidistribution of $\Gamma$ in bisectors defined using hypertubes (\cref{thm:BisectorEquidistribution}), generalizing \cite{CO23}. The asymptotics for counting Jordan and Cartan projections in hypertubes are deduced from these equidistribution results (\cref{thm:JordanCounting,thm:BisectorEquidistribution}). 
		
		\item In \cref{sec:ProofOfCorrelations}, we use the results of the previous sections to deduce \cref{thm:Counting} on asymptotics for linearly correlated Jordan and Cartan projections of $\Gamma$. The correlation theorems in the introduction are then deduced from \cref{thm:Counting}.
	\end{itemize}
	
	\section{Preliminaries on discrete subgroups of $G$}
	\label{sec:Preliminaries}
	
	Throughout the paper, let $G$ be a connected semisimple real algebraic group. Fixing a Cartan involution of the Lie algebra $\LieG$ of $G$,
	let $\LieG = \LieK \oplus \LieP$ be the  eigenspace decomposition corresponding to the eigenvalues $+1$ and $-1$ respectively. Let $K < G$ be the maximal compact subgroup whose Lie algebra is $\mathfrak{k}$. Let $\LieA \subset \LieP$ be a maximal abelian subalgebra and choose a closed positive Weyl chamber $\LieA^+ \subset \LieA$. We denote by $\Phi^+$ the set of all positive roots for $(\LieG, \LieA)$ with respect to the choice of $\fa^+$. Let $w_0 \in K$ be a representative of the longest Weyl element so that \begin{equation}\label{woo} \Ad_{w_0}(\LieA^+) = -\LieA^+.\end{equation} The map $\involution: \LieA^+ \to \LieA^+$ defined by $\involution(\w ) = -\Ad_{w_0}(\w )$ is called the opposition involution of $G$.  Let $A = \exp \LieA$ and  $A^+ = \exp \LieA^+$.
	
	Let $P$ be the minimal parabolic subgroup of $G$ given as $P=MAN$ where $M$ is the centralizer of $A$ in $K$ and $\log N$
	consists of all root subspaces corresponding to positive roots.
	The quotient
	$\cal{F}: = G/P $
	is called the {Furstenberg boundary} of $G$. By the Iwasawa decomposition $G=KAN$,
	we have $G/P\simeq K/M$. The {Iwasawa cocycle} $\sigma: G \times \Fboundary \to \LieA$ is the map which assigns to each $(g,kM) \in G \times \Fboundary$ the unique element $\sigma(g,kM) \in \LieA$ such that  $gk \in K\exp(\sigma(g,\xi))N$. The {$\LieA$-valued Busemann function} $\beta: \Fboundary \times G \times G \to \LieA$ is defined by
	\begin{equation}\label{buse}
		\beta_\xi(g_1,g_2) = \sigma(g_1^{-1}, \xi) - \sigma(g_2^{-1}, \xi)
	\end{equation}
	for all $g_1,g_2 \in G$ and $\xi \in \Fboundary$.

	\subsection*{Cartan and Jordan projections} 
	For $g \in G$, let  $\mu(g)$ denote the {Cartan projection} of $g$, i.e., $\mu(g) \in \LieA^+$ is the unique element in $\LieA^+$ such that 
	$g \in K\exp(\mu(g))K.$ 	Any non-trivial element $g\in G$ can be written as the commuting product $g=g_hg_e g_u$ where $g_h$ is hyperbolic, $g_e$ is elliptic and $g_u$ is unipotent. 
	The hyperbolic component $g_h$ is conjugate to a unique element $\exp \lambda(g) \in A^+$ and $\lambda(g)$ is called 
	the Jordan projection of $g$.
	When $\lambda(g)\in \inte \fa^+$, $g\in G$ is called {\it{loxodromic}} in which case $g_u$ is necessarily trivial and $g_e$ is conjugate to
	an element $m(g)\in M$ which is unique up to conjugation in $M$. We call its conjugacy class
	$[m(g)]\in [M]$ the {\it holonomy} of $g$. 
	
	\subsection*{Limit set, limit cone and holonomy group.}
	\label{subsec:LimitSetAndLimitCone} Let $\Ga<G$ be a Zariski dense discrete subgroup.
	Let $\La\subset \F$ denote the limit set of $\Ga$, which is the unique $\Gamma$-minimal subset of $\Fboundary$ \cite{Ben97}.
	The {\textit{limit cone}} $\limitcone = \limitcone_\Gamma $ of $\Gamma$ is the smallest closed cone containing the Jordan projections $\lambda(\Gamma)$. Benoist \cite{Ben97} proved that $\limitcone$ is convex and has non-empty interior. 
	The {\textit{holonomy group}} of $\Gamma$ is the closed subgroup $$M_\Gamma < M$$ generated by the holonomy classes $[m(\gamma)]$, $\gamma\in \Ga$.
	By \cite[Corollary 1.10]{GR07}, $M_\Gamma$ is a normal subgroup of $M$ of finite index. In general, $M_\Ga\ne M$ (e.g., Hitchin representations \cite[Theorem 1.5]{Lab06}).

	\subsection*{Growth indicators.} 
	The {growth indicator} $\growthindicator : \LieA^+ \to \R \cup \{-\infty\}$ of $\Gamma$ is defined by
	\begin{equation}\label{gr}
		\growthindicator(\w ) = \|\w \| \inf_{\text{open cones }\cal{C}\ni \w } \tau_\cal{C} \qquad \text{for all non-zero $\w  \in \LieA^+$}
	\end{equation}
	where  $\|\cdot\|$ is any norm on $\fa$ and $\tau_\cal{C}$ is the abscissa of convergence of the series $t \mapsto \sum_{\gamma \in \Gamma, \mu(\gamma) \in \cal{C}} e^{-t\|\mu(\gamma)\|}$. We set $\psi_\Ga(0)=0$.
    This definition is independent of the choice of a norm. We have that $\growthindicator$ is
    upper semicontinuous and concave. Moreover, it is positive on ${\interior\L}$ \cite[Theorem 4.2.2]{Qui02a}. These properties imply that $\growthindicator$ is continuous on $\inte\L$.

	Set $$\L^*=\{\psi\in \fa^*:\psi\ge 0 \text{ on $\cal L$}\}.$$ Then
	$$\inte \L^*=\{\psi\in \fa^*:\psi>0 \text{ on }\L-\{0\}\}.$$ For $\psi\in \L^*$, set
	$$\delta_\psi=\delta_{\Ga, \psi}=\limsup_{T\to \infty} \frac{\log \#\{\ga\in \Ga:\psi(\mu(\ga))<T\}}{T}\in [0, \infty].$$ A linear form $\psi \in \LieA^*$ is said to be tangent to $\growthindicator$ at $\v\in \fa^+-\{0\}$ if $\growthindicator \le \psi$ and $\growthindicator(\v)=\psi(\v).$ 
	
	\begin{theorem}[{\cite[Lemma 2.4 and Theorem 2.5]{KMO21}}]
		\label{thm:delta}
		For any non-zero $\psi\in \L^*$ with $\delta_\psi<\infty$,
		the linear form $\delta_\psi \psi $ is tangent to $\psi_\Ga$ and $\delta_\psi>0$.
		Moreover, if $\psi\in \inte\L^*$, then $\delta_\psi<\infty$.
	\end{theorem}

	\section{Critical vectors for proper linear maps}
	\label{sec:CriticalVector}
	Let $\Gamma$ be a Zariski dense Anosov subgroup of $G$ as defined
	in \eqref{al}. Then its limit cone $\L$ is contained in  $\interior\LieA^+ \cup \{0\}$ \cite[Proposition 4.6]{PS17}. Fix $1 \le d \le \rankG$ and a surjective linear map
	$$\bs{\varphi}:\LieA \to \R^d \text{ such that } \bs{\varphi}|_\limitcone \text{ is proper}.$$
	We note that $\bs{\varphi}|_\limitcone$ is a proper map if and only if $\ker \bs{\varphi}\cap \L=\{0\}$.
	Define the {\it{$\bs{\varphi}$-projection}} of $\limitcone$: $$\limitcone_{\bs{\varphi}}:=\bs{\varphi}(\limitcone) \subset \br^d.$$
	
	\begin{lemma}
		\label{lem:ProjectedCone}
		The set $\limitcone_{\bs{\varphi}}$ is a closed convex cone and $\interior\limitcone_{\bs{\varphi}} = \bs{\varphi}(\interior\limitcone)$.
	\end{lemma}
	
	\begin{proof}
		By \cite[Theorem 9.1]{Roc70}, if a linear map is nonzero on a given closed convex cone in $\LieA$ except at $0$, then the image of the given cone is a closed convex cone in $\R^d$. Hence by our hypothesis that $\bs{\varphi}|_\limitcone$ is proper,  $\limitcone_{\bs{\varphi}}$ is a closed convex cone. Moreover, by the Banach open mapping theorem which says that a surjective linear map is an open map, we have $\interior\limitcone_{\bs{\varphi}} = \bs{\varphi}(\interior\limitcone)$.
	\end{proof}

	We will use the following property of Zariski dense Anosov subgroups:
	\begin{theorem}
		\label{bij}
		\label{thm:Anosov} 
		We have:
		\begin{enumerate}
			\item The growth indicator $\psi_{\Gamma}$ is analytic and strictly concave on $\interior\limitcone$, except along rays emanating from the origin. It is also vertically tangent, meaning that if $\psi\in \fa^*$
			is tangent to $\psi_\Ga$ at $\v\in \L$, then $\v\in \inte \L$. 
			
			\item  For each vector $\v\in \inte\L$, there exists a unique linear form 
			\be\label{pv} \psi_\v\in \inte \L^*\ee  tangent to $\growthindicator$ at $\v$. Moreover, the map $\v\mapsto \|\v\| \psi_{\v}$ is a homeomorphism $\inte\L\to \inte\L^*$.
		\end{enumerate}
	\end{theorem}
	
	The first property in \cref{thm:Anosov} follows from  Quint's duality lemma \cite[Lemma 4.3]{Qui03} and the work of Potrie-Sambarino \cite[Proposition 4.11]{PS17}. The second property is deduced from the first property by using the derivative of $\psi_\Gamma$ to establish a continuous bijection $\interior\limitcone \to \interior\limitcone^*$. That it is indeed a homeomorphism can be proved either by using the upper semi-continuity of $\psi_\Ga$  (see \cite[Proposition 4.4]{LO20b}) or by
	applying the invariance of domain theorem which states that a continuous injection $\R^n \to \R^n$ is in fact a homeomorphism onto its image.
	
	\begin{definition}
		\label{def:CriticalVector}
		For a given vector $\bs{r}\in \br^d$, a vector $\v \in \interior\limitcone$ is called a {\it{$(\bs{\varphi},\bs{r})$-critical vector}} of $\Gamma$ if it satisfies
		$$\bs{\varphi}(\v) = \bs{r} \quad \text{ and } \quad \ker\bs{\varphi} < \ker\psi_\v.$$ 
	\end{definition}
	
	The following proposition plays a key role in our study of multiple correlation problem:
	\begin{proposition}
		\label{lem:CriticalVector}
		For any $\bs{r}\in \interior\limitcone_{\bs{\varphi}}$,
		there exists a unique $(\bs{\varphi},\bs{r})$-critical vector $\v^\star=\v^{\star}(\bp, \bs{r})\in \inte\L$ of $\Gamma$. 
	\end{proposition}

	The point of \cref{lem:CriticalVector} is that although the affine subspace $\bp^{-1}(\bs{r})$ has dimension $\rankG-d$, it contains a unique vector $\v^\star$ such that the subspace $\ker \psi_{\v^{\star}}$  (of dimension $\rankG-1$) contains the subspace $\ker\bp$ of dimension  $\rankG-d$. 
    
    Set $$ \sf{K}_{\bs{\varphi}} 
	=\{\v\in \inte\L: \ker\bs{\varphi} < \ker\psi_\v\}. $$

	Proposition \ref{lem:CriticalVector} follows from the following:
	\begin{lemma}
		\label{lem:Bijection}
		The map $\bs{\varphi}|_{\sf{K}_{\bs{\varphi}}}: \sf{K}_{\bs{\varphi}} \to \interior\limitcone_{\bs{\varphi}}$ 
		is a homeomorphism. In particular, for any $\bs{r} \in \interior \limitcone_{\bs{\varphi}}$,
		$\bs{\varphi}^{-1}(\bs{r})\cap \sf{K}_{\bs{\varphi}} =\{\v^\star\} $
		is the unique $(\bs{\varphi},\bs{r})$-critical vector of $\Gamma$. 
	\end{lemma} 
	
	\begin{proof} 
		To show the injectivity,  let  $\v_1, \v_2\in \sf{K}_{\bs{\varphi}}$.
		Suppose that $\v_1\ne  \v_2$  and  $\bp (\v_1)=\bp  (\v_2) $.  Then 
		$$\v_1- \v_2 \in \ker\bs{\varphi} <  \ker\psi_{\v_1} \cap \ker\psi_{\v_2}.$$
		That is, $\psi_{\v_i}(\v_1) = \psi_{\v_i}(\v_2)$ for $i=1,2$.
		First note that $\br \v_1 \ne \br \v_2$; otherwise $\v_1=t\v_2$ for some $t\ne 0$ and $\bp(\v_1)=\bp(\v_2)$ implies $t=1$ which contradicts $ \v_1\ne {\v_2}$. 
		By the strict concavity of $\psi_\Ga$ (\cref{thm:Anosov}), the linear form $\psi_{\v_2}$ is tangent to $\psi_\Ga$ only in the direction $\br \v_2\ne \br \v_1$. Therefore
		\begin{equation}
			\psi_{\v_2}(\v_1) > \psi_{\Gamma}(\v_1)   =\psi_{\v_1}(\v_1) = \psi_{\v_1}(\v_2).
		\end{equation} 
		A symmetric argument gives $\psi_{\v_1}(\v_2) > \psi_{\v_2}(\v_1)$ which yields a contradiction. This proves the injectivity. 
		
		By Theorem \ref{bij}(2),
	$\sf{K}_{\bs{\varphi}} $
    is homeomorphic to the subset
    $$\{\phi\in \inte\L^*: \ker\bs{\varphi} <\ker \phi\} =U\cap \inte \L^* $$
where $U= \{\phi\in \fa^*: \ker\bs{\varphi} <\ker \phi\}$, which is a $d$-dimensional subspace of $\LieA^*$.
Since $\ker\bs{\varphi}\cap \L=\{0\}$, there exists a linear form on $\LieA$ whose kernel contains $\ker\bs{\varphi}$ and is positive on $\L-\{0\}$. Therefore  ${U}\cap \interior\L^*\ne \emptyset$.  Since $\inte \L^*$ is an open convex cone of $\fa^*$, it follows that $ U\cap \inte \L^* $ is a non-empty open convex cone in ${U}$. Hence $\sf{K}_{\bs{\varphi}}$ is homeomorphic to $\br^d$.

Since $\bp$ is a continuous injective map on $\sf{K}_{\bs{\varphi}}$ and $\interior\limitcone_{\bs{\varphi}}$ is homeomorphic to $\br^d$, it is an open map by the invariance of domain theorem. In particular, the image $\bp (\sf{K}_{\bs{\varphi}})  $ is a non-empty open  subset of $\inte \L_{\bp}$.
		Since $\inte \L_{\bp}$ is connected,  surjectivity follows if we show that  $\bp (\sf{K}_{\bs{\varphi}})$
		is a closed subset of $\inte \L_{\bp}$.  Suppose that $\bp (\sf{K}_{\bs{\varphi}})$
		is not closed in  $\inte \L_{\bp}$. Then
		there exists a sequence $t_i>0$, unit vectors $\v_i\in \inte \L$ and $\w\in \inte \L-
		(\sf{K}_{\bs{\varphi}} +\ker \bp) $ such that
		$t_i \v_i\in \sf{K}_{\bs{\varphi}}$ and  $\bp(t_i\v_i)\to \bp(\w)\in \inte \L_{\bp} $ as $i\to \infty$. 
		Since $\v_i$ are unit vectors, we may assume, by passing to a subsequence, that
		$\v_i$ converges to some unit vector $\v\in \L$ and hence $t_i$ converges to some $t>0$. So $\bp (t\v)=\bp (\w)$.
		
        First consider the case when the sequence $\psi_{\v_i}\in \inte\L^*$ is bounded, and hence converges to some linear form $\psi\in \L^*$, by passing to a subsequence.
		It follows that $\psi$ is tangent to $\psi_\Ga$ at $\v\in \L$. 
		By the vertical tangency of $\psi_\Ga$ (\cref{thm:Anosov}), it follows that $\v\in \inte \L$.
		Since $\psi_{\v_i}\to \psi$ and $\ker\bs{\varphi} < \ker\psi_{\v_i}$, we have $\ker\bp < \ker\psi$ and
		hence $\v\in \sf{K}_{\bs{\varphi}}$. This is a contradiction since $\w\notin\sf{K}_{\bs{\varphi}} +\ker \bp$. Now suppose that the sequence $\psi_{\v_i}\in \inte\L^*$ is unbounded. Then
		by passing to a subsequence, we have a sequence $s_i\to 0$ such that
		$s_i\psi_{\v_i}$ converges to some linear form $\psi\in \L^*$. Since $\w\in \inte\L$,
		we have $\psi(\w)>0$.
		On the other hand, we have 
		$$ \psi(\v)=\lim_{i\to \infty} \psi_{\v_i}( s_i \v_i) = \lim_{i\to \infty} s_i \psi_\Ga
		( \v_i)=0$$  since $\psi_\Ga(\v_i)$ is bounded.
		Therefore $\psi(\w)=\psi(\v)=0$, yielding a contradiction.  This proves that $\bp (\sf{K}_{\bs{\varphi}})$
		is  closed in  $\inte \L_{\bp}$.
This  completes the proof.
	\end{proof}

	Here is another  characterization of the $(\bs{\varphi},\bs{r})$-critical vector of $\Gamma$:
	\begin{lemma}
		\label{lem:sup}
		For any $\bs{r}\in \interior\limitcone_{\bs{\varphi}}$, the $(\bs{\varphi},\bs{r})$-critical vector of $\Gamma$ is the unique vector $\v^\star\in \limitcone \cap \bp^{-1}(\bs{r})$ such that
		$$\psi_\Ga(\v^\star)= \sup_{\w  \in \limitcone, \, \bs{\varphi}(\w ) = \bs{r}}\psi_{\Gamma}(\w ).$$ 
	\end{lemma}
	\begin{proof}	Let $\v^\star$ be the $(\bs{\varphi},\bs{r})$-critical direction of $\Gamma$ as in \cref{lem:Bijection}. Suppose that there exists $\w  \in \limitcone$, $\w  \ne \v^\star$ such that $\bs{\varphi}(\w ) = \bs{r}$ and $\psi_\Gamma(\w ) \ge \psi_{\Gamma}(\v^\star)$. Since $\bs{\varphi}(\w ) = \bs{\varphi}(\v^\star)$, we have $\w  - \v^\star \in \ker\bs{\varphi} < \ker\psi_{\v^\star}$. By Theorem \ref{bij}(2),  we have $$\psi_{\v^\star} (\w ) > \psi_\Gamma(\w ) \ge \psi_{\Gamma}(\v^\star) = \psi_{\v^\star}(\v^\star) = \psi_{\v^\star}(\w )$$		which is a contradiction. Hence $\psi_\Ga(\v^\star)> \psi_\Ga(\w )$, proving the claim.\end{proof}
	
	\section{Hypertubes and multiple correlations}
	\label{sec:Hypertubes}
	Let $\Ga$ be a Zariski dense Anosov subgroup of a connected semisimple real algebraic group $G$. In our previous paper \cite{CO23}, we obtained {multiple} correlations for convex cocompact representations by counting Jordan and Cartan projections of $\Gamma$  lying in {\it tubes} of $\fa$. By a tube, we mean a subset of the form
	$\sQ +V$ where $V$ is a {\it line} in $\fa$ and $\sQ$ is a compact subset of a complementary\footnote{Two linear subspaces $V_1$ and $V_2$ of a vector space $V_0$ are complementary if $V_0 = V_1 \oplus V_2$.} subspace to $V$.
	
	In this section, we rephrase studying multiple correlations of the linearized Jordan (resp. Cartan) spectra of Anosov representations in terms of counting Jordan  (resp. Cartan) projections of $\Gamma$ that lie in what we call {\it{hypertubes}}.
	Hypertubes are generalizations of tubes; they are of the form
	$\sQ+V$ where $V$ is {\it any linear subspace} of $\fa$ which is not necessarily one dimensional and $\sQ$ is a compact subset of a complementary subspace  to $V$ in $\fa$. While there is a unique way to truncate a tube as the only way to go to infinity is along the line $V$, the choice of truncation of a hypertube has to be made carefully in order to be able to obtain desired counting results.
	
	\subsection*{Hypertubes} 
	\begin{definition}
		\label{def:Hypertube}
		A {\it hypertube} of the limit cone $\L$ is of the form
		$$\T=\T (\sQ ,V,\cal{C}) =(\sQ  + V) \cap\cal{C}$$
		where
		\begin{itemize}
			\item  $V$ is a non-zero linear subspace of $\fa$;
			\item $\sQ$ is a compact subset of a complementary subspace 
			to $V$ such that $\sQ$ has non-empty relative interior and has Lebesgue null boundary;
			\item   $\cal{C} \subset \interior\LieA^+$ is a closed convex cone with nonempty interior such that $\cal{C} \cap V \cap \interior\limitcone\ne \emptyset$. 
		\end{itemize}
	\end{definition}
	
	When $\dim V = 1$, hypertubes are called tubes.  
	
	\subsection*{Truncations of a hypertube} 
	Consider a hypertube 
	$$\T=\T(\sQ ,V,\cal{C}).$$
	Fix a vector $$ \v \in \cal{C} \cap V \cap  \interior\limitcone\;\; \text{ such that $\sQ\subset \ker\psi_\v$};$$ recall that	$\psi_\v\in \fa^*$ is the tangent form given in \eqref{pv}. 
	Noting that $\T= (\sQ +(V\cap \ker \psi_\v)+\br \v)\cap \cal C$, we  define $\v$-truncations of $\T$ as follows: for any $T>0$,
	$$\T_{T}(\v) = \left(\sQ + (V\cap \ker \psi_\v) +[0, T]\v\right)\cap\cal C.$$

	More generally, for any continuous function $b \in C(\sQ )$, 
	we consider the following family of $\v$-truncations of $\T$: for any $T >0$,
	\be\label{et} \T_{T,b}(\v) = \{\sf{q}  + \v' + t\v  \in \cal C: \sf{q}  \in \sQ , \, \sf{v}' \in V\cap \ker\psi_\v, \, 0 \le t \le T +b({\sf{q}})\};\ee 
	hence $\T_{T, 0}(\v) =\T_T(\v)$.
	
	\subsection*{Relation between multiple correlations and counting in hypertubes}
	We now explain the relationship between multiple correlations of Jordan (resp. Cartan) projections for Anosov representations and counting Jordan (resp. Cartan) projections in hypertubes of a single Anosov subgroup. 
	Let $\Sigma$ be a finitely generated group.
	Let $\bs{\rho}=(\rho_i:\Sigma \to G_i)_{1\le i \le d}$ be a $d$-tuple of Anosov representations such that $\bs{\rho}(\Sigma)$ is Zariski dense in the product $\prod_{i=1}^d G_i$ and $(\psi_i:\LieA_i \to \R)_{1 \le i \le d}$ be a $d$-tuple of linear forms such that $\psi_i>0$ on $\limitcone_{\rho_i(\Sigma)} -\{0\}$ for each $1\le i\le d$. Then in particular, $\G=\bs{\rho}(\Sigma)$ is a Zariski dense Anosov subgroup of $G=\prod_{i=1}^d G_i$,
    $\bs{\varphi} = (\varphi_1,\dots,\varphi_d):\fa \to \R^d $
	is a surjective  linear map
    and $\bs{\varphi}|_{\limitcone_{\bs{\rho}(\Sigma)}}$ is a proper map. 
	
	Indeed we now describe a more general situation 
	for any Zariski dense Anosov subgroup $\Ga$ of a connected semisimple real algebraic group $G$. Let $d\ge 2$ and $\varphi_1,\dots,\varphi_d $ be linear forms on $\fa$ so that the linear map
	$$\bs{\varphi} = (\varphi_1,\dots,\varphi_d):\fa \to \R^d $$
	is surjective and $\bs{\varphi}|_{\limitcone}$ is proper where $\L=\L_\Ga$ is the limit cone of $\Ga$. By \cref{lem:ProjectedCone}, $\limitcone_{\bs{\varphi}} = \bs{\varphi}(\limitcone)$ is a closed convex cone with nonempty interior.
	Fix $\bs{r} \in \interior\limitcone_{\bs{\varphi}}$ and $\e_1,\dots,\e_d > 0$. We are interested in finding asymptotics for
	$$\#\{[\gamma] \in [\Gamma] : \, \lambda(\gamma) \in \bv^{-1}(\prod_{i=1}^d[r_iT,  r_iT+\e_i])\}
	\quad\text{ and }$$
	$$\#\{\gamma \in \Gamma : \, \mu(\gamma) \in \bv^{-1}(\prod_{i=1}^d[r_iT,  r_iT+\e_i])\}.$$
	That is, we wish to count the Jordan and Cartan projections of $\Gamma$ that lie in the set $ \bs{\varphi}^{-1}(\prod_{i=1}^{d}[r_iT,r_iT+\e_i])$.

	Using the $(\bs{\varphi},\bs{r})$-critical vector of $\Gamma$,  we are able to describe the shape of this set in terms of truncations of a hypertube; this description is crucial for the anaylsis of this paper and is the content of the following lemma.
	
	$$\text{Let $\v^{\star} \in \interior\limitcone$ be the $(\bs{\varphi},\bs{r})$-critical vector of $\Gamma$};$$ that is,
	$\bs{\varphi}(\v^\star) = \bs{r}$ and $\ker\bs{\varphi^\star} $ is a subspace of $ \ker\psi_\v$. Existence and uniqueness of $\v^{\star}$  was proved in \cref{lem:CriticalVector}.
	
	\begin{lemma} [Basic Lemma] \label{lem:Hypertube}\label{lem:Truncations}  
		Let $\cal C$ be a closed convex cone of $\interior\LieA^+$  such that $\limitcone\subset \inte \cal{C}$. \begin{enumerate}
			\item  The following set  is a hypertube of $\limitcone$:
			$$\T=\bigcup_{T>0} 
			\{\u  \in \cal{C}: \bs{\varphi}(\u) \in \prod_{i=1}^{d}[r_iT,r_iT+\e_i]\} .$$ 
			More precisely, if  $W < \ker\psi_{\v^{\star}}$ is a complementary  subspace to $\ker\bs{\varphi}$ and   $ \sQ =W\cap \bp^{-1}(\prod_{i=1}^{d}[0,\e_i]\})$, then 
			$$\T = \T(\sQ ,\R\v^{\star}\oplus\ker\bs{\varphi},\cal{C}) = (\sQ  + (\R\v^{\star}\oplus\ker\bs{\varphi})) \cap \cal{C}. $$
			\item   There exist functions $b_1,b_2 \in C(\sQ )$ such that for all $T > 0$, we have
			$$\cal{C}\cap\bs{\varphi}^{-1}(\prod_{i=1}^{d}[r_iT,r_iT+\e_i]) = \T_{T,b_1}(\v^{\star}) - \T_{T,b_2}(\v^{\star}).$$
		\end{enumerate}
	\end{lemma}
	\begin{proof}
		Fix a subspace $W < \ker\psi_{\v^{\star}}$ complementary to $\ker\bs{\varphi}$; this exists since $\ker\bs{\varphi} $ is a subspace of $ \ker \psi_{\v^{\star}}$ of codimension $d-1$.
		Since $W \cap \ker\bs{\varphi} = \{0\}$,   $\bs{\varphi}|_{W}$ is a proper map.  Therefore
		$ \sQ =W\cap \bp^{-1}(\prod_{i=1}^{d}[0,\e_i]\})$ is a compact subset. It is clear that $\sQ$ has Lebesgue null boundary and $\sQ$ has nonempty interior. To see that $\T$ is a hypertube of $\limitcone$, observe that
		\begin{align*}
			\T & = \bigcup_{T>0}\{\u  \in \cal{C}: \bs{\varphi}(\u) \in \prod_{i=1}^{d}[r_iT,r_iT+\e_i]\}
			\\
			& = \bigcup_{T>0}\{\w  + \sf{v}' + T\v^{\star} \in \cal{C}: \bs{\varphi}(\w) \in \prod_{i=1}^{d}[0,\e_i], \w  \in W, \sf{v}' \in \ker\bs{\varphi}\}	 	
			\\
			& = (\{\w \in W: \bs{\varphi}(\w) \in \prod_{i=1}^{d}[0,\e_i]\}+ (\R\v^{\star}\oplus\ker\bs{\varphi})) \cap \cal{C}
			\\
			& = (\sQ + (\R\v^{\star}\oplus\ker\bs{\varphi})) \cap \cal{C}.
		\end{align*}
		To recover the set $\cal{C}\cap\bs{\varphi}^{-1}(\prod_{i=1}^{d}[r_iT,r_iT+\e_i])$ as a difference of $\v^{\star}$-truncations, consider the $d$-dimensional parallelepiped
		$$B = \{\u  \in \R\v^{\star}\oplus W: \bs{\varphi}(\u) \in \prod_{i=1}^d[0,\e_i]\}.$$
		Observe that for all $\sf{q}  \in \sQ $, the set $\{t \in \R: \sf{q}  + t\v^{\star} \in B\} \subset \R$ is an interval $[b_2(\sf{q} ), b_1(\sf{q} )]$ since any line in $\R\v^{\star}\oplus W$ that intersects $B$ does so in a single segment (or possibly a single point).  This gives us continuous functions $b_1$ and $b_2$ on $\sQ$. Hence 
		\begin{align*}
			& \T_{T,b_1}(\v^{\star}) - \T_{T,b_2}(\v^{\star}) 
			\\
			& = \{\sf{q}  + \v' + t\v^{\star}  \in \T: \sf{q}  \in \sQ , \, \sf{v}' \in \ker\bs{\varphi}, \, T + b_2(\sf{q} ) \le t \le T +b_1({\sf{q} })\} 
			\\
			& = ((T\v^{\star}+ B) +\ker\bs{\varphi})\cap\cal{C}
			\\ 
			& = \{\u \in\cal{C}:\bs{\varphi}(\u) \in \prod_{i=1}^{d}[r_iT,r_iT+\e_i]\},
		\end{align*} where $\bp(\v^{\star})=\bs{r}$ was used in the last identity.
	\end{proof}

	By \cref{lem:Truncations}, to prove \cref{thm:Correlations} it suffices to find asymptotics for counting Jordan and Cartan projections of $\Gamma$ that lie in $\v$-truncations of $\T$ which is the goal of the next section. 
	
	\section{Equidistribution with respect to hypertubes}
	\label{sec:EquidistributionStatements}
	Let $\Ga$ be a Zariski dense Anosov subgroup of a connected semisimple real algebraic group $G$.
	In our paper \cite{CO23}, we proved joint counting and equidistribution theorems of Jordan and Cartan projections of $\Ga$
	with respect to  tubes of the limit cone $\limitcone=\L_\Ga$. The goal of this section is to generalize these results to hypertubes. 
	
	For this entire section,  fix a hypertube
	$$\T=\T (\sQ_0 ,V,\cal{C}) =(\sQ_0  + V) \cap\cal{C}$$
	of $\L$, where $\sQ_0, V, $ and $\cal C$ are as in
	\cref{def:Hypertube}. 
	Fix a vector $$\v \in \cal{C} \cap V \cap  \interior\limitcone,$$ which exists by definition of a hypertube. 
	Let $ W< \ker\psi_\v$ be a complementary subspace to $V \cap \ker\psi_\v$.  We then have $\LieA = V \oplus W$ and $\ker \psi_\v= (V\cap \ker \psi_\v) \oplus W$. Let $\sQ \subset W$ be the image of $\sQ_0 $ under the projection
	$ V \oplus W\to W$. Then $\sQ$ spans $W$ and in particular, 
	$$\T = \T(\sQ,V,\cal{C}).$$

	Fix a function $b\in C(\sQ)$ and recall the associated $\v$-truncations of $\T$: \be\label{et2} \T_{T,b}(\v) = \{\sf{q}  + \v' + t\v  \in \T: \sf{q}  \in \sQ , \, \sf{v}' \in V\cap \ker\psi_\v, \, 0 \le t \le T +b({\sf{q}})\}.\ee 
	
	\begin{remark}
		In the setting of Basic \cref{lem:Hypertube}, $V = \br \v \oplus \ker\bp$, $\sQ= W\cap \bp^{-1}(\prod_{i=1}^{d}[0,\e_i])$ and $\v$ is the
		$(\bp, \bs{r})$-critical vector of $\Ga$. Therefore the counting and equidistribution statements proven for the truncations $\T_{T,b}$ in this section apply to $\T_{T,b_i}$ in \cref{lem:Hypertube}.
	\end{remark}
	
	\subsection*{Equidistribution of cylinders with respect to hypertubes}
	We will fist state a joint equidistribution theorem on nontrivial closed $AM$-orbits in $\Gamma \ba G$ and their holonomies whose periods are ordered by the $\v$-truncations $\T_{T,b}(\v)$ (\cref{thm:JointEquidistribution}). Recall that any element of $\Ga$ of infinite order is a loxodromic element by the Anosov property \cite[Corollary 3.2]{GW12}.
	Let $\primGamma$ denote the set of primitive loxodromic elements in $\Gamma$. 
	For each conjugacy class $[\gamma]\in [\primGamma]$, consider the closed $A$-orbit given by
	$$ C_{\gamma}=\Gamma\ba \Gamma gAM \subset \Ga\ba G/M $$ 
	where $g\in G$ is such that
	$ \gamma\in g (\inte A^+) Mg^{-1}$. This is well-defined independent of the choice of $g$ and the choice of $\ga$ in its conjugacy class.
	Each $C_\gamma$ is  homeomorphic to a cylinder $\S^1 \times \R^{\rankG -1}$  \cite[Lemma 4.14]{CF23}. 
	For a continuous function $f$ on $\Ga\ba G/M$, the integral $\int_{C_\gamma}f$  is computed with respect to the measure on $C_\gamma$ induced by the Haar measure of  $A$. 
	
	There exists a unique probability $(\Ga, \psi_\v)$-conformal measure, say, $\nu_\v$, on $\F$, that is,
	for any $\gamma \in \Gamma$ and $\xi \in \Fboundary$,
	$$\frac{d\gamma_*\nu_\v}{d\nu_\v}(\xi) = e^{\psi_\v(\beta_\xi(e,\gamma))}$$
	where $\gamma_*\nu(Q) = \nu(\gamma^{-1}Q)$ for any Borel subset $Q \subset \Fboundary$  and $\beta$ is the Busemann map defined in \eqref{buse} (\cite{LO20b} and \cite{LO22}); moreover $\nu_\v$ is supported on $\La$. 
	For all $g \in G$, let
	\be  \label{eqn:FurstenbergBoundaryNotation}
	g^+ = gP \in \Fboundary \quad \text{ and }\quad 
	g^- = gw_0P \in \Fboundary\ee 
	where $w_0$ is the longest Weyl element as in \eqref{woo}.
	There is a unique open $G$-orbit in $\Fboundary \times \Fboundary$ given by
	$		\Fboundary^{(2)} = G.(e^+,e^-) \subset \Fboundary\times\Fboundary$.
	The {Hopf parametrization} is a diffeomorphism $G/M \to \Fboundary^{(2)} \times \LieA$ defined by
	\begin{equation}\label{hopf}
		gM \mapsto (g^+, g^-, \w =\beta_{g^+}(e, g)) \quad\text{for all $g\in G$}.
	\end{equation}
	The associated BMS measure ${\BMS}$ on $G/M$
	is then given by $$d{\sf m}^{\op{BMS}}_\v(gM) = e^{\psi_\v \left(\beta_{g^+}(e,g)\right) + \psi_{\i(\v)} \left(\beta_{g^-}(e,g)\right)} \, d\nu_\v(g^+) \, d\nu_{\i(\v)}(g^-) \, d\w .$$
	This being left $\Ga$-invariant, it descends to an $A$-invariant measure  on $\Ga \ba G /M$, which we denote by the same notation $\BMS$ by abuse of notation (see \cite[Lemma 3.6]{ELO20} for details). Note that it is
    supported on the subset $\Ga\ba (\La^{(2)}\times \fa ) \subset \Ga \ba G /M$, where $\La^{(2)}=(\La\times \La ) \cap \F^{(2)}$.

Consider the $\Ga$-action on $\La^{(2)}\times \br$
given by $$\ga (\xi, \eta, t)= (\ga\xi, \ga \eta, t +\psi_\v (\beta_{\ga \xi}(o, \ga o))$$ for $\ga\in \Ga$
and $(\xi, \eta, t)\in \La^{(2)}\times \br$. This is a proper discontinuous and cocompact action
by \cite[Proposition 4.1]{BCLS} and \cite[Theorem 3.5]{CS23}.  Setting $\mathcal X_\v=\Ga\ba (\La^{(2)}\times \br)$, the product measure 
$$d{\tilde m}_{\mathcal X_\v}(\xi, \eta, t) = e^{\psi_\v \left(\beta_{\xi}(e,g)\right) + \psi_{\i(\v)} \left(\beta_{\eta}(e,g)\right)} \, d\nu_\v(\xi) \, d\nu_{\i(\v)}(\eta) \, dt $$
on $\La^{(2)}\times \br$,
where $g\in G$ is any element such that $g^+=\xi$ and $g^-=\eta$, induces a finite measure 
\begin{equation}
   \label{mcal} dm_{{\cal{X}_\v}} \end{equation}
on $\mathcal X_\v$
(see \cite[Theorem 4.8]{LO20b} for the statement and references).
The projection $\Ga \ba \La^{(2)}\times \fa \to \mathcal X_\v$ induced by
    $(\xi, \eta, \u)\mapsto (\xi, \eta, \psi_\v(\u))$
    is a principle $\ker \psi_\v$-bundle.
    The Bowen-Margulis-Sullivan measure $d\m_\v^{\op{BMS}}$  is a product
	\begin{equation}
		\label{eqn:BMSProduct}
		d\BMS = dm_{{\cal{X}_\v}} \, d\u 
	\end{equation}
	where  $d\u $ denotes the appropriately normalized Lebesgue measure on $\ker\psi_\v$ (\cite[Proposition 3.5]{Sam15}, \cite[Corollary 4.9]{LO20b}).
	
	Set 
	$$ \rank := \rankG,  \; d := \rank + 1 - \dim V \;\; \text{and } \;\; \delta_\v := \psi_\Gamma(\v).$$
	Let $C_{\mathrm{c}}(\Ga \ba G/M)$ denote the set of continuous compactly supported functions on $\Ga \ba G/M$ and $\mathrm{Cl}(M)$ denote the set of continuous class functions on $M$. We now state the equidistribution of closed cylinders  $C_\ga$ with the Jordan projection $\lambda(\gamma)$ restricted to the $\v$-truncations of $\T$:
	\begin{theorem}
		\label{thm:JointEquidistribution}
		For any $f \in C_{\mathrm{c}}(\Ga \ba G/M)$ and $\phi \in \mathrm{Cl}(M)$, we have as $T\to \infty$,
		$$\sum_{[\gamma] \in [\primGamma],\, \lambda(\gamma) \in \T_{T,b}(\v)} \int_{C_\ga} f \; \cdot \phi(m(\gamma)) \sim c\cdot \BMS(f)\cdot\int_{M_\Ga}\phi \, dm \cdot \frac{e^{\delta_\v T}}{T^{(d-1) /2}}$$
		where $dm$ denotes the Haar probability measure on $M$, and $c=c(\T,\v,b) >0$ is a constant defined in \eqref{eqn:c}.
	\end{theorem}
	
	\cref{thm:JointEquidistribution} was previously proved in \cite{CO23} for  tubes rather than hypertubes. 
	The remainder of this section is devoted to the proof of \cref{thm:JointEquidistribution}. We will often refer the reader to \cite[Sections 5]{CO23}, when the proofs are similar.
	
	\subsection*{Local mixing}
	Let $dx$ denote the right $G$-invariant measure on $\Gamma \backslash G$ induced by the Haar measure on $G$. 	The following theorem  was proved in \cite[Theorem 3.4]{ELO22b} using local mixing of the BMS measure \cite[Theorem 1.3]{CS23}.

	\begin{theorem}
		\label{thm:DecayofMatrixCoefficients}
		There exist a constant $\kappa_{\sf{v}} >0$ and an inner product $\langle \cdot, \cdot \rangle_*$ on $\LieA$ such that for any $\u  \in \ker\psi_\v$ and $\phi_1, \phi_2 \in C_{\mathrm{c}}(\Gamma \backslash G)$, we have
		\begin{multline*}
			\lim_{t \to +\infty} t^{\frac{\rank - 1}{2}}e^{(2\varrho - \psi_\v)(t\sf{v} + \sqrt{t}\u )} \int_{\Gamma \backslash G} 	\phi_1(xa_{t\sf{v} + \sqrt{t}\u }) \phi_2(x) \, dx \\
			=\frac{\kappa_{\sf{v}}e^{-I(\u )}}{|m_{\cal{X}_\v}|}  \sum_{Z} 	\BR\bigr|_{Z\check{N}}(\phi_1)\cdot 	\BRstar\bigr|_{ZN}(\phi_2)
		\end{multline*}
		where the sum is taken over all $A$-ergodic components $Z$ of $\BMS$,  $\BR$ and $\BRstar$ are the Burger-Roblin measures
		which are respectively right $N$ and  $\check{N}$-invariant where $\check{N} = w_0Nw_0^{-1}$ and 
		\begin{equation*}		\label{eqn:IDefinition}
			I(\u ) = \langle \u , \u \rangle_* - \frac{\langle \u , \sf{v} \rangle_*^2}{\langle \sf{v}, \sf{v}\rangle_*}
		\end{equation*} 
		
		Moreover, there exist $\eta_{\sf{v}}>0$ and $s_{\sf{v}}>0$ such that for all $\phi_1, \phi_2 \in C_{\mathrm{c}}(\Gamma \backslash G)$, there exists $D_\v>0$ depending continuously on $\phi_1$ and $\phi_2$ such that for all $(t,\u ) \in (s_{\sf{v}},\infty) \times \ker\psi_\v$ such that $t\sf{v}+\sqrt{t}\u  \in \LieA^+$, we have
		$$\left|t^{\frac{\rank - 1}{2}}e^{(2\varrho - \psi_\v)(t\sf{v} + \sqrt{t}\u )} \int_{\Gamma \backslash G} \phi_1(xa_{t\sf{v} + \sqrt{t}\u }) \phi_2(x) \, dx\right| \le D_{\sf{v}}  e^{-\eta_{\sf{v}} I(\u )}.
		$$
	\end{theorem}

	\subsection*{An integral asymptotic}
	The following integral of the multiplicative coefficients of \cref{thm:DecayofMatrixCoefficients} over the $\v$-truncations $\T_{T,b}(\v)$ will play an important role in computing the asymptotic in \cref{thm:JointEquidistribution}:
	\begin{equation}
		\label{eqn:L}
		L(\T_{T,b}(\v)):=\frac{\kappa_{\v}}{|m_{\cal{X}_\v}|}\int_{t\v+\sqrt{t}\u \in \T_{T,b}}e^{\delta_\v t}e^{-I(\u)} \, dt \, d\u
	\end{equation}
	where $\kappa_\v $, $I(\u)$ and $m_{\cal{X}_\v}$ are as in \cref{thm:DecayofMatrixCoefficients} and \eqref{eqn:BMSProduct}.

	Denote by $d\v'$ and $d\w $ the Lebesgue measures on $V\cap \ker\psi_\v$ and $W$ so that the Lebesgue measure $d\u $ on $\ker\psi_\v =  (V\cap \ker\psi_\v) \oplus W$ satisfies	\begin{equation}	
		\label{eqn:Lebesgue}		
		d\u  = d\v' \, d\w .	
	\end{equation} 
	Consider the following constant:
	\begin{equation}
		\label{eqn:c}  
		c(\T,\v,b) = \frac{\kappa_\v}{\delta_\v|m_{\cal{X}_\v}|}\int_{V\cap \ker\psi_\v}e^{-I(\v')} \, d\v' \int_{\sQ } e^{\delta_\v b(\w)} \, d\w.
	\end{equation}
	
	\begin{lemma}
		\label{lem:Asymptotic}
		We have as $T \to \infty$,
		$$L(\T_{T,b}(\v)) \sim c(\T,\v,b)\frac{e^{\delta_\v T}}{T^{(d-1)/2}}.$$
	\end{lemma}

	In order to prepare for the proof of this lemma, for $T > 0$,
	$\v' \in V \cap \ker\psi_\v$ and $\w \in W$, let 
	\begin{align*}
		R_T(\v',\w) & = \{ t\ge 0: t\v+\sqrt{t}(\v' + \w) \in \T_{T,b}(\v)\}
		\\
		& = \{0 \le t \le T+b(\sqrt{t}\w): \sqrt{t}\w \in \sf{Q}, t\v+\sqrt{t}(\v' + \w) \in \cal{C}\}.
	\end{align*}
	Define $f_T:  \ker\psi_\v  \to \br$ by
	$$f_T(\v',\w ) = \frac{1}{e^{\delta_\v T}}e^{- I(\v' +\w /\sqrt{T})} \int_{R_T(\v',\w /\sqrt{T})} e^{\delta_\v t}\, dt$$
	for $(\v', \w)\in (V\cap \ker\psi_\v)\oplus W$.
	
	\begin{lemma} For all $\v' \in V\cap \ker\psi_\v$ and $\w \in W$, we have
		\begin{equation*}
			\lim_{T\to\infty}f_T(\v',\w )= \begin{cases}  \frac{1}{\delta_\v}e^{- I(\v')}e^{\delta_\v b(\w )} &\text{if $ \w \in \inte \sQ $}  \\ 0 & \text{if $\w \in W- \sQ $} .\end{cases} 
		\end{equation*}  
	\end{lemma}\begin{proof}
		Fix $\v' \in V\cap \ker\psi_\v$ and $\w \in W$. Let
		$$\cal{J}(\v',\w ) = \inf\{t \ge 0 : t\v + \sqrt{t}\v' + \w  \in \cal{C}\}.$$ 
		Since $\cal{C}$ is a cone and $\v \in \interior\cal{C}$, if $t_0 \in \cal{J}(\v',\w )$, then $t \in \cal{J}(\v',\w )$ for all $t \ge t_0$. Then we have
		\begin{align*}
			R_T(\v',\w /\sqrt{T}) & = \{\cal{J}(\v',\sqrt{\tfrac{t}{T}}\w ) \le t \le T+b(\sqrt{\tfrac{t}{T}}\w ): \sqrt{\tfrac{t}{T}}\w  \in \sQ \}
			\\
			& \supset \{\sup_{\sf{q}  \in \sQ }\cal{J}(\v',\sf{q} ) \le t \le T+b(\sqrt{\tfrac{t}{T}}\w ): \sqrt{\tfrac{t}{T}}\w  \in \sQ \}
		\end{align*}	
		where $\sup_{\sf{q}  \in \sQ }\cal{J}(\v',\sf{q} )$ is independent of $T$. Then setting 
		$$R_T(\w/\sqrt{T}) =  \{0 \le t \le T+b(\sqrt{\tfrac{t}{T}}\w ): \sqrt{\tfrac{t}{T}}\w  \in \sQ \},$$
		we have 
		$$	\lim_{T\to\infty}\frac{1}{e^{\delta_\v T}} \int_{R_T(\v',\w /\sqrt{T})} e^{\delta_\v t}\, dt = 	\lim_{T\to\infty}\frac{1}{e^{\delta_\v T}} \int_{R_T(\w /\sqrt{T})} e^{\delta_\v t}\, dt.$$
		By the same computations done for $R_T(u/\sqrt{T})$ in \cite[Lemma 5.3]{CO23}, we obtain
		\begin{equation*}
			\lim_{T\to\infty}\frac{1}{e^{\delta_\v T}} \int_{R_T(\w /\sqrt{T})} e^{\delta_\v t}\, dt = \begin{cases}  \frac{1}{\delta_\v}e^{\delta_\v b(\w )} &\text{if $ \w \in \inte \sQ $}  \\ 0 & \text{if $\w \in W- \sQ $} .\end{cases} 
		\end{equation*} 
		Recalling the definition of $I$ from \eqref{eqn:IDefinition}, it is easy to check that 
		\begin{equation*}
			I(\v'+\w /\sqrt{T}) = \tfrac{1}{T}I(\w ) + \tfrac{2}{\sqrt{T}}(\langle \v', \w \rangle_*-\tfrac{\langle \v', \v\rangle_*\langle \w ,\v\rangle_*}{\langle\v,\v\rangle_*}) + I(\v').
		\end{equation*}
		Putting these together completes the proof of the claim. \end{proof}
	
	\subsection*{Proof of \cref{lem:Asymptotic}}
	Fix $T > 0$. For $\u  \in \ker\psi_\v$, we write $\u  = \sf{v}'+ \w /\sqrt{T}$ where $\v' \in V\cap \ker\psi_\v$ and $\w  \in W$. Then we have
	\begin{equation}
		\label{eqn:f}
		\begin{aligned}[b]
			& \int_{t\v + \sqrt{t}\u  \in \T_{T,b}}e^{\delta_\v t}e^{-I(\u )}dt \, d\u  
			\\
			= &\frac{1}{T^{(d-1)/2}}\int_{V\cap \ker\psi_\v}\int_{W}\int_{t \in R_T(\v',\w/\sqrt{T})}e^{\delta_\v t}e^{-I(\v'+\w/\sqrt{T} )}dt \, d\w \, d\v'
			\\
			= &\frac{e^{\delta_\v T}}{T^{(d-1)/2}}\int_{V\cap \ker\psi_\v}\int_{W}f_T(\v',\w) \, d\w \, d\v'.
		\end{aligned}
	\end{equation}
	Set
	$$\sf{A}_T= \int_{V\cap \ker\psi_\v}\int_{\sQ } f_T(\v',\w)\, d\w \, d\v';$$
	$$\sf{B}_T= \int_{V\cap \ker\psi_\v}\int_{\mathrm{hull} (\sQ \cup \{0\}) -\sQ } f_T(\v',\w)\, d\w \, d\v';$$
	$$\sf{C}_T= \int_{V\cap \ker\psi_\v}\int_{\w\notin \mathrm{hull}(\sQ  \cup \{0\})} f_T(\v',\w)\, d\w \, d\v'$$
	where $\mathrm{hull}(\sQ\cup\{0\})$ denotes the convex hull of $\sQ \cup \{0\}$.
	By the hypothesis that $\partial \sQ$ has measure zero, we have
	\be\label{ft}\int_{V\cap \ker\psi_\v}\int_{W}f_T(\v',\w) \, d\w \, d\v'= \sf{A}_T + \sf{B}_T + \sf{C}_T.\ee 
	
	\noindent{\bf Asymptotics of $\sf{A}_T$ and $\sf{B}_T$.}
	Let $\w \in\mathrm{hull} (\sQ \cup \{0\})$. In this case,
	since $R_T(\v',\w) \subset [0, T+\max b]$ for all $\v' \in V\cap \ker\psi_\v$ and $T \ge 1$, we have
	\begin{align*}
		f_T(\v',\w) & \le \frac{1}{e^{\delta_\v T}}e^{- I(\v'+\w/\sqrt{T} )} \int_0^{T+\max b} e^{\delta_\v t}\, dt 
		\\
		& \le \frac{1}{\delta_\v}e^{\delta_\v \max b}e^{- \inf_{\sf{q} \in\mathrm{hull} (\sQ \cup \{0\})}I(\v'+\sf{q})}.
	\end{align*}
	Note that, as a function of $\v' \in V\cap \ker\psi_\v$, $\inf_{\sf{q} \in\mathrm{hull} (\sQ \cup \{0\})}I(\v' + \sf{q})$ is radially increasing with at least quadratic rate: for all $r>1$, we have
	$$\inf_{\sf{q} \in\mathrm{hull} (\sQ \cup \{0\})}I(r\v'+\sf{q}) = \inf_{\sf{q} \in\mathrm{hull} (\sQ \cup \{0\})}r^2I(\v'+\sf{q}/r) \ge r^2\inf_{\sf{q} \in\mathrm{hull} (\sQ \cup \{0\})}I(\v' + \sf{q}).$$
	Hence the function $\v' \mapsto e^{- \inf_{\sf{q} \in\mathrm{hull} (\sQ \cup \{0\})}I(\v' + \sf{q})}$ is $L^1$-integrable on $V\cap \ker\psi_\v$. Therefore we can apply the Lebesgue dominated convergence theorem to obtain
	\begin{align*}
		\lim_{T\to\infty}\sf{A}_T & = \int_{V\cap \ker\psi_\v}\int_{\sQ } \lim_{T\to\infty}f_T(\v',\w)\, d\w \, d\v' 
		\\ & =  \frac{1}{\delta_\v}\int_{V\cap \ker\psi_\v}e^{- I(\v')}\, d\v'\int_{\sQ }e^{\delta_\v b(\w)} \, d\w \\&= \frac{|m_{\cal{X}_\v}|}{\kappa_\v}c(\T,\v,b)
	\end{align*}
	and
	$$\lim_{T\to\infty}\sf{B}_T = \int_{V\cap \ker\psi_\v}\int_{\mathrm{hull} (\sQ \cup \{0\}) -\sQ } \lim_{T\to\infty}f_T(\v',\w)\, d\w \, d\v' =  0.$$
	
	\noindent{\bf Asymptotic of $\sf{C}_T$.}
	Let $\w \notin\mathrm{hull} (\sQ \cup \{0\})$and $\v' \in V\cap \ker\psi_\v$.
	Observe that 
	$$j_{\w}: =
	\sup \{0\le s : s\w \in \sQ \} <1$$
	and hence
	\begin{align*}
		& R_T(\v'/\sqrt{T},\w/\sqrt{T}) 
		\\
		= & \{0 \le t \le T+b(\sqrt{t/T}\w): \sqrt{t/T}\w \in \sQ , t\v+\sqrt{t/T}(\v' + \w) \in \cal{C}\} 
		\\
		\subset & [0,j_\w^2T]. 
	\end{align*}
	Let 
	$$g(\v',\w)={ 2\sqrt{\delta_\v I(\v'+\w)(1-j_{\w}^2)}} .$$
	Then for all $T\ge 1$, we have
	\begin{align*}
		f_T(\v'/\sqrt{T},\w) & \le \frac{1}{e^{\delta_\v T}}e^{- I(\v'/\sqrt{T}+\w/\sqrt{T})} \int_{R_T(\v'/\sqrt{T},\w/\sqrt{T})} e^{\delta_\v t}\, dt 
		\\
		& \le \tfrac{1}{\delta_\v}e^{- I(\v'+\w)/T}e^{-\delta_\v T(1-j_{\w}^2)} 
		\\
		& \le \tfrac{1}{\delta_\v}e^{- g(\v',\w)}
	\end{align*}
	where the last inequality uses the fact that $a+b \ge 2\sqrt{ab}$ for all $a,b \ge 0$.
	Note that for all $r>1$,
	\begin{align*}
		g(r\v',r\w) &={2 \sqrt{\delta_\v I(r\v' + r\w)(1-j_{r \w}^2)}} 
		\\
		& =2r\sqrt{\delta_\v I(\v'+\w)(1-j_{\w}^2/r^2)} \ge r g(\v',\w),
	\end{align*}
	i.e., the function $g(\v',\w)$ is radially increasing with at least a linear rate.  Hence the function $(\v',\w)\mapsto  e^{-g(\v',\w)}$ is $L^1$-integrable on $(V\cap\ker\psi_\v) \times (W -\mathrm{hull}(\sQ  \cup \{0\}))$.
	
	Then performing a change of variable and applying the Lebesgue dominated convergence theorem, we obtain
	\begin{align*}
		\lim_{T\to\infty} \sf{C}_T & = \lim_{T\to\infty}\int_{V\cap\ker\psi_\v}\int_{\w\notin \mathrm{hull} (\sQ \cup \{0\})}\frac{1}{T^{(\rank - d)/2}}f_T(\v'/\sqrt{T},\w) \, d\w \, d\v'
		\\
		& = \int_{V\cap\ker\psi_\v}\int_{\w\notin \mathrm{hull} (\sQ \cup \{0\})}\lim_{T\to\infty}\frac{1}{T^{(\rank - d)/2}}f_T(\v'/\sqrt{T},\w) \, d\w \, d\v'
		\\
		&= \int_{V\cap\ker\psi_\v}\int_{\w\notin \mathrm{hull} (\sQ \cup \{0\})}0 \, d\w \, d\v'= 0.
	\end{align*} 
	By \eqref{eqn:L}, \eqref{eqn:f} and \eqref{ft}, we obtain
	\begin{align*}
		\lim_{T \to \infty}\frac{T^{(d-1)/2}}{e^{\delta_\v T}}L(\T_{T,b}(\v)) & =\frac{\kappa_{\v}}{|m_{\cal{X}_\v}|}\lim_{T\to\infty}(\sf{A}_T+\sf{B}_T+\sf{C}_T)
		\\
		& = c(\T,\v,b)
	\end{align*}
	and this finishes the proof of \cref{lem:Asymptotic}.

	\subsection*{Counting in $\check{N}AMN$-coordinates.}
	\label{subsec:CountinginNAMNCoordinates}
	
	Fix bounded Borel sets 
	$$\check{\Xi} \subset\check{N}=w_0Nw_0^{-1},\;\; \Xi \subset N,\;\; \Theta \subset M$$ with non-empty relative interiors and null boundaries: $$\nu_\v(\partial (\check{\Xi} e^+))=\nu_{\involution(\v)}(\partial (\Xi^{-1}e^-))=\vol_M(\partial\Theta)=0.$$
	
	For $T > 0$, let  
	$$S_{T,b} = \check{\Xi}\exp(\T_{T,b}(\v)) 	\Theta \Xi \subset \check{N}AMN.$$

	The first counting result towards \cref{thm:JointEquidistribution} is an asymptotic for $\#(\Gamma \cap S_{T,b})$ for which we use \cref{lem:Asymptotic}. For simplicity, we state the asymptotic in the case that $M_\Ga = M$. When $M_\Ga \ne M$, the fact that $\BMS$ has $[M:M_\Ga]$ many $A$-ergodic components makes the statement of asymptotics more involved  as in \cref{thm:DecayofMatrixCoefficients}, but this can be handled in the same way as in \cite[Proposition 5.12]{CF23}.
	
	\begin{proposition} 
		\label{prop:STAsymptotic}
		We have as  $T\to\infty$,
		$$\#(\Gamma \cap S_{T,b}) \sim c(\T,\v,b) \tilde{\nu}_\v(\check{\Xi})\tilde{\nu}_{\involution(\v)}(\Xi^{-1})\vol_M(\Theta)\frac{e^{\delta_\v T}}{T^{(d-1)/2}}$$
		where $\tilde{\nu}_\v$ and $\tilde{\nu}_{\involution(\v)}$ are measures on $\check{N}$ and $N$, respectively defined by
		\begin{align*}
			& d\tilde{\nu}_\v(h) = e^{\psi_\v(\beta_{h^+}(e,h))}  \, d\nu_\v(h^+) ;
			& d\tilde{\nu}_{\involution(\v)}(n) = e^{(\psi_\v\circ\involution)(\beta_{n^-}(e,n))}\, d\nu_{\involution(\v)}(n^-)
		\end{align*}
		for $h \in \check{N}$ and $n \in N$. 
	\end{proposition}
	
	\begin{proof} The main new technical ingredient of this proof is \cref{lem:Asymptotic}. Using this,
		we explain how the proof \cite[Proposition 5.1]{CO23} can be adjusted for our setting. Given a bounded Borel subset $B$ of $G$, define the counting function $F_B:\Ga\ba G \times\Ga\ba G \to \N$ by
		$$F_B(x,y) = \sum_{\gamma \in \Gamma} 1_B(g^{-1}\gamma h) \quad\text{for $x=\Gamma g, y=\Gamma h\in \Gamma \ba G$}.$$ 
		The quantity $ F_{S_{T,b}}([e],[e]) = \#\Gamma \cap S_{T,b}$ can be approximated above and below as follows. For $\e>0$, let $G_\e$ denote the $\e$-neighborhood of identity in $G$ and similarly for other subgroups of $G$. For any $\e > 0$, let
		\begin{align*}
			& S_{T,b,\e}^- = \bigcap_{g_1,g_2 \in G_\e}g_1S_{T,b}g_2; & S_{T,b,\e}^+ = \bigcup_{g_1,g_2 \in G_\e}g_1S_{T,b}g_2,
		\end{align*}
        and 
		$$\psi_\varepsilon \in C_c(G) \text{ such that } \psi_\e\ge 0, \;\; \supp \psi_\varepsilon \subset G_\varepsilon \text{ and } \int_G \psi_\varepsilon \, dg = 1.$$ 
		Define $\Psi_\varepsilon \in C_c^\infty(\Gamma \backslash G) $ by $$ \Psi_\varepsilon([g]) = \sum_{\gamma \in \Gamma} \psi_\varepsilon(\gamma g).$$ 
		Then for any $\e>0$, we have 
		\begin{multline}
			\label{eqn:int}
			\int_{\Gamma \backslash G \times \Gamma\backslash G}F_{S_{T,b,\varepsilon}^-}(x,y)\Psi_\e(x)\Psi_\e(y) \, dx \, dy 
			\\
			\le F_{S_{T,b}}([e],[e]) \le \int_{\Gamma \backslash G \times \Gamma\backslash G}F_{S_{T,b,\varepsilon}^+}(x,y)\Psi_\e(x)\Psi_\e(y) \, dx \, dy
		\end{multline}
		where the integrals are taken with respect to the Haar measure on $\Gamma \ba G$.
		
		We now need to approximate the sets $S_{T,b,\e}^\pm$ for $\e$ small with approximations of the same product form as $S_{T,b}$ using $\e$-neighborhoods of $\T_{T,b}(\v)$. It is important that these approximations are well-behaved and indeed, for tubes, these approximations are again tubes. However, unlike a tube,   the $\varepsilon$-neighborhood of a hypertube is not again a hypertube as we have defined. Indeed, the $\varepsilon$-neighborhood of $\sf{Q}+V$ in $\LieA$ is of the same form but the $\e$-neighborhood of the hypertube $\T=(\sf{Q}+V)\cap\cal{C}$ is not a hypertube since the $\e$-neighborhood of a cone is not a cone. However, the asymptotic in \cref{lem:Asymptotic} does not depend on the cone $\cal{C}$ which contains $\v$ in its interior so it will suffice for us to approximate $\T$ using hypertubes that use cones which approximate $\cal{C}$ as in the following lemma whose proof is similar to \cite[Lemma 5.2]{CO23}:
		
		\begin{lemma}
			\label{lem:STpmBounds} 
			For all sufficiently small $\e > 0$,
			there exist  hypertubes $\T^\pm_\e := \T(\sQ^\pm_\e,V,\cal{C}^\pm_\e)$, with $\v$-truncations $\T^\pm_{T,b_\e^\pm}(\v)$ of $\T^\pm_\e$ and
			Borel subsets $\check{\Xi}_{\e}^- \subset \check{\Xi} \subset \check{\Xi}_{\e}^+$ of $\check N$, $\Xi_{\e} ^- \subset \Xi \subset \Xi_{\e} ^+$  of $N$ and $\Theta_{\e}^- \subset \Theta \subset \Theta_{\e}^+$  of $M$ satisfying the following:
			\begin{enumerate}	\item \label{itm:4.3.1} for all $T>0$,
				\begin{multline*}
					\check N_{O(\e)}\check{\Xi}_{\e}^- \exp(\T_{T,b_\e^-}^-(\v)) M_{O(\e)} \Theta^-_{\e} \Xi_{\e} ^-N_{O(\e)}\subset S_{T,b,\e}^-  \\ \subset S_{T,b,\e}^+
					\subset \check N_{O(\e)}\check{\Xi}_{\e}^+ \exp(\T_{T,b_\e^+}^+(\v)) M_{O(\e)} \Theta^+_{\e} \Xi_{\e} ^+N_{O(\e)}
				\end{multline*}
				where the inclusions hold up to some bounded subset of $G$ that does not depend on $\e$ and $T$;
				\item \label{itm:4.3.2} an $O(\e)$-neighborhood of $\T_{T,b^-_\e}(\sQ^-_\e,V,\cal{C})(\v)$ 
				(resp. $\T_{T, b}(\v)$) contains $\T_{T,b}(\v)$ (resp.  $\T_{T,b^+_\e}(\sQ^+_\e,V,\cal{C})(\v)$);
				
				\item \label{itm:4.3.3} $\nu_\v((\check{\Xi}_{\e}^+ -\check{\Xi}_{\e}^-) e^+ ) \to 0$, $\nu_{\involution(\v)}((\Xi_{\e} ^+ -\Xi_{\e} ^-) e^- ) \to 0$ and $\vol_M(\Theta_\e^+ - \Theta_\e^-) \to 0$ as $\e \to 0$.
			\end{enumerate}
		\end{lemma}
		
		Returning to \eqref{eqn:int}, using local mixing (\cref{thm:DecayofMatrixCoefficients}) to extract the main terms of the integrals in \eqref{eqn:int} for large $T$ and the same computations as in \cite[Proposition 5.1]{CO23} and approximating $S_{T,b,\varepsilon}^\pm$ using \cref{lem:STpmBounds}, we obtain
		\begin{align*}
			& \int_{\Gamma \backslash G \times \Gamma\backslash G}F_{S_{T,b,\varepsilon}^\pm}(x,y)\Psi_\e(x)\Psi_\e(y) \, dx \, dy 
			\\ &
			\sim (1+O(\varepsilon))L(\T_{T,b_\e^\pm}^\pm(\v))\tilde{\nu}_\v(\check{\Xi})\tilde{\nu}_{\involution(\v)}(\Xi ^{-1})\vol_M(\Theta)\\ 
			&  \sim (1+O(\varepsilon))  c(\T^\pm,\v,b_\e^\pm)\tilde{\nu}_\v(\check{\Xi})\tilde{\nu}_{\involution(\v)}(\Xi ^{-1})\vol_M(\Theta)\frac{e^{\delta_\v T}}{T^{(d-1)/2}} \quad \text{ as } T \to \infty,
		\end{align*}  
		where we used \cref{lem:Asymptotic} in the last asymptotic.	Taking $\e \to 0$  completes the proof.
	\end{proof}

	\subsection*{Counting via flow boxes.}
	We use {\it{flow boxes}}:
	\begin{definition}[$\varepsilon$-flow box at $g_0$]
		\label{def:FlowBox}
		Given $g_0 \in G$ and $\varepsilon > 0$, the {\it{$\varepsilon$-flow box at $g_0$}} is defined by 
		$$\cal{B}(g_0,\varepsilon)=g_0(\check{N}_\varepsilon N \cap N_\varepsilon \check{N} AM)M_\varepsilon A_\varepsilon.$$	
		We denote the projection of $\cal{B}(g_0,\varepsilon)$ into $\Gamma \backslash G/M$ by $\tilde{\cal{B}}(g_0,\varepsilon)$. 
	\end{definition}
	
	For $g_0\in G$ and $T, \varepsilon >0$, we denote
	\begin{equation*}
		\label{eqn:VTDefinition}
		\cal{V}_{T,b}(g_0,\varepsilon) = \cal{B}(g_0,\varepsilon)\T_{T,b}(\v)\Theta\cal{B}(g_0,\varepsilon)^{-1};
	\end{equation*}
	$$\cal{W}_{T,b}(g_0,\varepsilon)=\{gamg^{-1}:g\in\cal{B}(g_0,\varepsilon),am\in \T_{T,b}(\v) \Theta\}.$$

	The next proposition is an asymptotic for $\#(\Gamma \cap \cal{V}_{T,b}(g_0,\varepsilon))$:
	
	\begin{proposition}
		\label{prop:CountingInVT} 
		Let $g_0 \in G$.
		For all sufficiently small $\varepsilon>0$, we have  
		\begin{multline*}
			\#(\Gamma\cap\cal{V}_{T,b}(g_0,\varepsilon))
			\\
			=c(\T,\v,b)\left( \frac{\BMS(\tilde{\cal{B}}(g_0,\varepsilon))}{b_{\rank}(\varepsilon)}\vol_M(\Theta\cap M_\Ga) (1+O(\varepsilon)) +o_T(1)\right)\frac{e^{\delta_\v T}}{ T^{(d-1)/2}}
		\end{multline*}
		where $b_r(\varepsilon)$ denotes the volume of the Euclidean $\rank$-ball of radius $\varepsilon$.
	\end{proposition} 
	
	\begin{proof}
		The proof of \cref{prop:CountingInVT} is the same as in \cite[Proposition 5.6]{CO23} after approximating $\cal{V}_{T,b}(e,\varepsilon)$ with $S_{T,b}$ (cf. \cite[Lemma 5.7]{CO23}) and replacing \cite[Proposition 5.1]{CO23} with \cref{prop:STAsymptotic} for the asymptotic.
	\end{proof}
	
	\subsection*{Proof of \cref{thm:JointEquidistribution}}
	
	For $T>0$, let $\eta_T$ denote the Radon measure on $\Ga \ba G/M \times [M]$ defined by the left hand side in \cref{thm:JointEquidistribution}, that is, for $f \in \mathrm{C}_{\mathrm{c}}(\Ga\ba G/M)$ and $\varphi \in \mathrm{Cl}(M)$, let 
	
	\begin{equation*}
		\eta_{T}(f\otimes\varphi)=\sum_{
			[\gamma] \in [\primGamma],\, \lambda(\gamma) \in \T_{T,b}(\v)} \int_{C_\ga} f \cdot \varphi(m(\gamma)).
	\end{equation*} 
	By \cite[Lemma 6.3]{CF23} for all sufficiently large $T$, we have
	$$\eta_{T}(\tilde{\cal{B}}(g_0,\varepsilon)\otimes\Theta)=b_{\rank}(\varepsilon)\cdot\#(\primGamma\cap\cal{W}_{T,b}(g_0,\varepsilon)).$$
	By the Anosov property, we have $\limitcone \subset \interior\LieA^+\cup\{0\}$ so we can apply a closing Lemma \cite[Lemma 2.7]{CF23} and approximate $\cal{W}_{T,b}(g_0,\varepsilon)$ using $\cal{V}_{T,b}(g_0,\varepsilon)$ (cf. \cite[Lemma 6.5]{CO23}) and use \cref{prop:CountingInVT} to obtain (cf. \cite[Proposition 5.13]{CO23}):
	\begin{multline*}
		\eta_{T}(\tilde{\cal{B}}(g_0,\varepsilon)\otimes\Theta)
		\\
		\sim c(\T,\v,b) \left( \BMS(\tilde{\cal{B}}(g_0,\varepsilon))\vol_M(\Theta) (1+O(\varepsilon)) \right)\frac{e^{\delta_\v T}}{T^{(d-1)/2}} \text{ as } T \to \infty.
	\end{multline*}
	A partition of unity argument now completes the proof of \cref{thm:JointEquidistribution}.

	\subsection*{Counting Jordan projections in hypertubes}	
	
	We deduce from \cref{thm:JointEquidistribution} the following asymptotic for counting Jordan projections of $\Gamma$ in $\T_{T,b}$ with holonomies in a given subset of $M$ which we will apply to correlations in \cref{sec:ProofOfCorrelations}: 
	
	\begin{theorem}
		\label{thm:JordanCounting}
		For any conjugation invariant Borel subset $\Theta \subset M$ with $\vol_M(\partial\Theta) = 0$, we have as $T \to \infty$,
		\begin{multline*}  
		\#\{[\gamma] \in [\Gamma]: \lambda(\gamma) \in \T_{T,b}(\v), \, m(\gamma) \in \Theta\} \\ \sim c(\T,\v,b) \cdot |m_{\cal{X}_\v}| \cdot \vol_M(\Theta\cap M_\Gamma)\cdot \frac{e^{\delta_\v T}}{T^{(d+1)/2}}
        \end{multline*}
        where $m_{{\mathcal X}_{\v}}$ is the finite measure in \eqref{mcal}.
	\end{theorem}
	
	\begin{proof} 
		Using the same arguments as in \cite[Section 5]{CO23}, the following renormalized version of \cref{thm:JointEquidistribution} can be deduced:
		For any $f \in C_\mathrm{c}(\Ga\ba G/M)$ and for any $\phi \in \mathrm{Cl}(M)$, we have as $T\to \infty$,
		\begin{multline}
			\label{re}
			\sum_{[\gamma] \in [\primGamma],\, \lambda(\gamma) \in \T_{T,b}(\v)} \frac{1}{\psi_\v(\lambda(\gamma))}\int_{C_\ga} f \; \cdot \phi(m(\gamma)) 
			\\
			\sim c(\T,\v,b)\cdot \BMS(f)\cdot\int_{M_\Ga}\phi \, dm \cdot \frac{e^{\delta_\v T}}{T^{(d-1) /2}}.
		\end{multline}
		\cref{thm:JordanCounting} follows from \eqref{re} using similar arguments as in \cite[Corollary 4.3]{CO23} which we outline.
		Using the product structure $\supp\BMS \cong \cal{X}_\v \times \ker\psi_\v$, choose $$f= \mathbbm{1}_{{\cal{X}_\v}} \otimes f_1 \in C_{\mathrm{c}}(\supp\BMS)$$ 
		where $f_1 \in C_{\mathrm{c}}(\ker\psi_\v)$ with $\int f_1(\u)\, d\u=1$. Then
		$$\BMS(f) = |m_{\cal{X}_\v}|\int f_1(\u)\, d\u = |m_{\cal{X}_\v}|.$$
		Using the product structure of $\BMS$ for Anosov subgroups \eqref{eqn:BMSProduct}, we have 
		$$\int_{C_\ga} f = \psi_\v(\lambda(\ga))\int f_1(\u)\, d\u = \psi_\v(\lambda(\ga))$$
		for every $[\ga] \in [\primGamma]$. By applying \eqref{re} to this function $f$ and $\phi = \mathbbm{1}_{\Theta}$ (using a standard partition of unity argument),  we obtain 
		\begin{multline}
			\label{eqn:Prim}
			\#\{[\gamma]\in[\primGamma]:\lambda(\gamma) \in \T_{T,b}(\v), \, m(\gamma) \in \Theta\} 
			\\
			\sim c(\T,\v,b) \cdot |m_{\mathcal{X}_\v}| \cdot \vol_M(\Theta\cap M_\Gamma)\cdot \frac{e^{\delta_\v T}}{T^{(d+1)/2}}.
		\end{multline}
		By an elementary argument, \eqref{eqn:Prim} remains true if $\primGamma$ is replaced with $\Gamma$. 
	\end{proof}

	\subsection*{Equidistribution in bisectors}
	For a pair of Borel subsets $\Xi_1, \Xi_2 \subset K$  such that $\Xi_1$ and $\Xi_2^{-1}$ are right $M$-invariant, consider the following subsets of $G=KA^+K$ where the $A^+$ component is given by the truncations $\T_{T,b}(\v)$: for $T>0$, set
	$$B(\T_{T,b}(\v), \Xi_1,\Xi_2) = \Xi_1\exp(\T_{T,b}(\v))\Xi_2 \subset G.$$
	\cref{thm:BisectorEquidistribution} describes the equidistribution of $\Gamma$ in {\textit{bisectors}} $B(\T_{T,b}(\v), \Xi_1,\Xi_2)$. 	 	
	
	\begin{theorem}
		\label{thm:BisectorEquidistribution}
		Suppose $\nu_\v(\partial\Xi_1) = \nu_{\involution(\v)}(\partial\Xi_2^{-1}) = 0$. Then we have 
		$$\# (\Gamma \cap B(\T_{T,b}(\v), \Xi_1,\Xi_2)) \sim c(\T,\v,b)\cdot\nu_\v(\Xi_1)\cdot\nu_{\involution(\v)}(\Xi_2^{-1})\cdot\frac{e^{\delta_\v T}}{T^{(d-1)/2}} \quad \text{as } T\to \infty.$$
		In particular,
		$$\#\{\gamma \in \Gamma: \mu(\gamma) \in \T_{T,b}(\v)\} \sim c(\T,\v,b)\cdot\frac{e^{\delta_\v T}}{T^{(d-1)/2}} \quad \text{as } T\to \infty.$$   
	\end{theorem}
	
	\begin{proof}
		\cref{thm:BisectorEquidistribution} was proved in \cite[Theorem 6.1]{CO23} for tubes rather than hypertubes. The proof of \cref{thm:BisectorEquidistribution} is the same as in those cases using hypertubes instead and the asymptotic from \cref{lem:Asymptotic}.
	\end{proof}

	\section{Multiple correlations for Anosov representations}
	\label{sec:ProofOfCorrelations}
	\label{sec:Applications}
	
	In this section, we obtain multiple correlations for Jordan and Cartan projections of a Zariski dense Anosov subgroup, using the Basic \cref{lem:Hypertube} and counting results for Jordan and Cartan projections in hypertubes (\cref{thm:JordanCounting,thm:BisectorEquidistribution}).  
	We then deduce \cref{thm:Correlations} from this result and highlight particular instances of our correlation theorems for geometric structures arising from examples of Anosov representations.

	\subsection*{Linearly correlated Jordan and Cartan projections}
	
	Let $\Gamma$ be a Zariski dense Anosov subgroup of $G$ with limit cone $\limitcone$. Fix any integer $d$ with $1 \le d \le \rankG$ and any surjective linear map 
	$$\bs{\varphi} = (\varphi_1,\dots,\varphi_d):\LieA \to \R^d \text{ such that } \bs{\varphi}|_{\limitcone} \text{ is proper}.$$ 
	By \cref{lem:ProjectedCone}, $\limitcone_{\bs{\varphi}} = \bs{\varphi}(\limitcone)$ is a closed convex cone with nonempty interior. 
	
	\begin{theorem}
		\label{thm:Counting}
		For any $\bs{r}=(r_1, \cdots, r_d) \in \interior\limitcone_{\bs{\varphi}}$, any $\e_1,\dots,\e_d > 0$ and any conjugation invariant Borel subset $\Theta < M_\Gamma$ with $\vol_{M_\Gamma}(\partial\Theta)=0$, we have as $T \to \infty$,
		\begin{multline}\label{f}
			\#\{[\gamma] \in [\Gamma] : \, \bv(\lambda(\gamma)) \in \prod_{i=1}^d[r_iT,  r_iT+\e_i],\;  m(\gamma) \in \Theta\} 
			\\
			\sim c\cdot\frac{e^{\psi_\Gamma(\v^{\star}) T}}{T^{(d+1)/2}}\vol_{M_\Gamma}(\Theta)
		\end{multline}
		and
		\be\label{s} \#\{\gamma \in \Gamma : \, \bv(\mu(\gamma)) \in \prod_{i=1}^d[r_iT,  r_iT+\e_i]\}  \sim c' \cdot c\cdot\frac{e^{\psi_\Gamma(\v^{\star})T}}{T^{(d-1)/2}}\ee 
		where $c = c(\Gamma, \bs{\varphi}, \bs{r}, \e_1,\dots,\e_d)$ and $c' = c'(\Gamma,\bs{\varphi},\bs{r})$ are positive constants and $\v^{\star} \in \interior\limitcone$ is the unique $(\bs{\varphi},\bs{r})$-critical vector of $\Gamma$. 
	\end{theorem}
	\begin{proof}
		Let $\v^{\star} \in \interior\limitcone$ be the unique $(\bs{\varphi},\bs{r})$-critical vector of $\Gamma$ given by Proposition \ref{lem:CriticalVector}. Since $\ker\bp<\ker \psi_{\v^\star}$, we can choose a subspace $W < \ker\psi_\v$ such that $W\oplus\ker\bs{\varphi} = \ker\psi_{\v^{\star}}$. Since $\L\subset \inte\fa^+\cup\{0\}$, we can choose a closed convex cone $\cal{C} \subset \interior\LieA^+$ such that $\limitcone \subset \interior\cal{C}$.
		Set $ \sQ =W\cap \bp^{-1}(\prod_{i=1}^{d}[0,\e_i]\})$.
		Then by \cref{lem:Hypertube}, we have
		$$\T=\bigcup_{T>0} 
		\{\u  \in \cal{C}: \bs{\varphi}(\u) \in \prod_{i=1}^{d}[r_iT,r_iT+\e_i]\} =\T(\sQ ,\R\v^{\star}\oplus\ker\bs{\varphi},\cal{C}) , $$ 
		and  for some $b_1, b_2\in C(\sQ)$,
		$$\cal{C}\cap\bs{\varphi}^{-1}(\prod_{i=1}^{d}[r_iT,r_iT+\e_i]) = \T_{T,b_1}(\v^{\star}) - \T_{T,b_2}(\v^{\star})\;\;\text{for all $T>0$}.$$
		Therefore 
		\begin{multline*}
			\#\{[\gamma] \in [\Gamma] : \, \bv(\lambda(\gamma)) \in \prod_{i=1}^d[r_iT,  r_iT+\e_i],\;  m(\gamma) \in \Theta\}
			\\
			= \#\{[\gamma] \in [\Gamma]: \lambda(\gamma) \in \T_{T,b_1}(\v^{\star}), \, m(\gamma) \in \Theta\} 
			\\
			- \#\{[\gamma] \in [\Gamma]: \lambda(\gamma) \in \T_{T,b_2}(\v^{\star}), \, m(\gamma) \in \Theta\};
		\end{multline*}
		Hence the first asymptotic \eqref{f} follows from \cref{thm:JointEquidistribution} with
    \begin{equation}\label{cintro}
      c= c(\T, \v^{\star}, b_1)  -c(\T, \v^{\star}, b_2).   \end{equation}
		
		Since $\inte \cal{C}$ contains $\L$ which is the asymptotic cone of $\mu(\Ga)$ \cite{Ben97},
		we have $\mu(\Ga)\subset \cal C$ except for finitely many points and hence
		we have
		\begin{multline*}
			\#\{\gamma \in \Gamma : \, \bv(\mu(\gamma)) \in \prod_{i=1}^d[r_iT,  r_iT+\e_i]\}
			\\=\#\{\gamma \in \Gamma: \mu(\gamma) \in \T_{T,b_1}(\v^{\star})\} - \#\{\gamma \in \Gamma: \mu(\gamma) \in \T_{T,b_2}(\v^{\star})\}+O(1).
		\end{multline*}
		Therefore the second asymptotic \eqref{s} follows from \cref{thm:BisectorEquidistribution} with
         \begin{equation}\label{cpintro}
      c'= \frac{1}{|m_{\mathcal X_{\v}}|} .   \end{equation}
      \end{proof}

	\begin{remark}\label{rem:Zariski}
		\begin{enumerate}
			\item 
			Properness of $\bs{\varphi}|_{\limitcone}$, or equivalently the property that $\ker\bs{\varphi} \cap \limitcone = \{0\}$, is a necessary condition for the intersection $\mu(\Gamma) \cap S_T$ to be finite for each $T>0$, where $S_T= \{\w  \in \LieA^+ : \, r_iT \le \varphi_i(\w ) \le r_iT+\e_i, 1 \le i \le d\}$. Indeed, since $S_T$ is invariant under translation by $\ker\bs{\varphi}$ and of bounded distance from $\ker\bs{\varphi}$, we have $\#\mu(\Gamma) \cap S_T$ is finite if and only if $\ker\bs{\varphi}$ has trivial intersection with the asymptotic cone of $\mu(\Gamma)$, i.e., $\bs{\varphi}|_{\limitcone}$ is proper. 
			\item Even when $G$ is the product of rank one groups, \cref{thm:Counting} is more general than our previous results in \cite{CO23} since $\bs{\varphi}$ need not be injective.

		\end{enumerate}
	\end{remark}

	\subsection*{Proof of \cref{thm:Correlations}}
	Note that $\Gamma:=\bs{\rho}(\Sigma)$ is a Zariski dense Anosov subgroup of $\prod_{i=1}^d G_i$, $(\varphi_1,\dots,\varphi_d)|_{\limitcone_{\bs{\rho}(\Sigma)}}$ is a proper map, and $M_{\bs{\rho}}=M_\Ga$. It is easy to see that $\limitcone_{\bs{\varphi}}$ in \cref{thm:Correlations} coincides with the $\bs{\varphi}$-projection of $\limitcone_{\bs{\rho}(\Sigma)}$. 
	Hence the asymptotics in \cref{thm:Correlations} is a particular instance of \cref{thm:Counting} with $\Gamma = \bs{\rho}(\Sigma)$ and we have $\delta_{\bs{\rho},\bs{\varphi}}(\bs{r}) = \psi_{\bs{\rho}(\Sigma)}(\v^\star)$ where $\v^\star$ is the $(\bs{\varphi},\bs{r})$-critical vector of $\bs{\rho}(\Sigma)$ and $\psi_{\bs{\rho}(\Sigma)}$ is the growth indicator of
	${\bs{\rho}(\Sigma)}$.
	It remains to prove the upper bound for $\delta_{\bs{\rho},\bs{\varphi}}(\bs{r}) = \psi_{\bs{\rho}(\Sigma)}(\v^\star)$. Since $\varphi_i$ is positive on $\limitcone_{\rho_i(\Sigma)} - \{0\}$, when we view $\varphi_i$ as an element of $\LieA^*$ in the natural way, $\varphi_i$ is positive on $\limitcone_{\bs{\rho}(\Sigma)}$. By \cref{thm:delta}, we have $\delta_{\bs{\rho}(\Sigma),\varphi_i} = \delta_{\rho_i(\Sigma),\varphi_i}$ and $\delta_{\rho_i(\Sigma),\varphi_i}\varphi_i$ is tangent to $\psi_{\bs{\rho}(\Sigma)}$. 
	By \cite[Theorem 1.4]{KMO21}, we have
	$$\psi_{\bs{\rho}(\Sigma)}(\v^\star) \le \min_{1 \le i \le d}\delta_{\rho_i(\Sigma),\varphi_i}\varphi_i(\v^\star) = \min_{1 \le i \le d}\delta_{\rho_i(\Sigma),\varphi_i}r_i$$
	and if all $\delta_{\rho_i(\Sigma),\varphi_i}\varphi_i(\v^\star)$ are equal, then the inequality is strict.
	\begin{remark}
		Clearly ${\bs{\rho}}(\Sigma)$ being Zariski dense in $\prod_{i=1}^d G_i$ implies that $\rho_i(\Sigma)$ is Zariski dense in $G_i$ for all $i$ and that for all $i \ne j$, $\rho_i \circ \rho_j^{-1}: \rho_j(\Sigma) \to \rho_i(\Sigma)$ does not extend to a Lie group isomorphism $G_j \to G_i$. In fact, the converse holds provided that $G_i$ is simple for all $i$ (cf. \cite[Lemma 4.1]{KO22}). 
	\end{remark}
	
	\subsection*{Proof of \cref{thm:Hilbert}}
	The Hilbert metric $d^{\op{H}}_\Omega$ on a properly convex domain  $\Omega \subset \R\P^2$ is defined as 
	\be\label{hilbert} d^{\op{H}}_\Omega(x,y) = \frac{1}{2}\log\left(\frac{\|w-y\| \cdot\|x-z\|}{\|w-x\|\cdot\|y-z\|}\right)\quad\text{for $x,y\in \Omega$}\ee 
	where $w,z \in \partial\Omega$ are such that $w,x,y,z$ are colinear in that order and $\|\cdot\|$ is a Euclidean norm on an affine chart $\mathbb{A}$ containing $\Omega$. The geodesics are precisely the projective lines in $\Omega$ and the isometries are precisely the elements of $\PGL_3\R$ preserving $\Omega$. 
	Identify the positive Weyl chamber of $\PGL_3\R$ with $\{(x_1,x_2,x_3) \in \R^3: x_1+x_2+x_3 =0, \,  x_1 \ge x_2 \ge x_3\}$.
	A basic fact from convex projective geometry is that the Hilbert length is given by 
	$$\ell^{\op{H}}_{\rho_i}(\sigma) = \tfrac{1}{2}(\lambda_1(\rho_i(\sigma))-\lambda_3(\rho_i(\sigma)))$$
	where $\lambda_j(g)$ denotes the logarithm of the $j$'th largest modulus of an eigenvalue of $g \in \PGL_3\R$. Another fact that the only automorphisms of $\PGL_3\R$ are the inner automorphisms and their composition with the contragradient involution. By \cite[Theorem 1.4]{DGK17}, $\rho$ is also an Anosov representation. Then by \cref{rem:Zariski}, $\bs{\rho}(\Sigma)$ is Zariski dense in $(\PGL_3\R)^d$. 
	\cref{thm:Hilbert} now follows by applying \cref{thm:Correlations} with $\varphi_i(x_1,x_2,x_3) = \tfrac{1}{2}(x_1-x_3)$ for all $i$. 
	
	\subsection*{Correlations for convex cocompact hyperbolic manifolds}
	Consider a $k$-tuple $\bs{\rho} = (\rho_1,\dots,\rho_k )$
	of Zariski dense convex cocompact representations of a finitely generated group $\Sigma$ into $\SO^\circ(n,1)$ for some $n\ge 2$ and $k\ge 2$.  We further assume that for all $i \ne j$, $\rho_i$ and $\rho_j$ are not conjugates. Let $\ell_i(\sigma)$ denote the length of the closed geodesic in the hyperbolic manifold $\rho_i(\Sigma) \ba \H^n$ corresponding to the conjugacy class $[\rho_i(\sigma)]$. \cref{thm:Counting} has the following application:
	
	\begin{theorem}
		\label{thm:Hyperbolic}
		Let $\bs{\phi}:\R^k \to \R$
		be a linear map such that
		$\{\bs{\phi}(\ell_1(\sigma), \dots, \ell_k(\sigma))\in \R :\sigma \in \Sigma\}$
		contains both negative and positive values. Then there exists $\delta > 0$ such that for any $\e > 0$, there exists a constant $c>0$ such that as $T \to \infty$,
		$$\{[\sigma] \in [\Sigma]: 0 \le \bs{\phi}(\ell_1(\sigma),\dots,\ell_k(\sigma)) \le \e, \, T \le \ell_k(\sigma) \le T + \e\} \sim c \cdot \frac{e^{\delta T}}{T^{3/2}}.$$
	\end{theorem}
	
	\begin{proof}
		Let $\bs{\varphi}: \R^k \to \R^2$ be given by $\bs{\varphi}(x_1,\dots,x_k) = (\bs{\phi}(x_1,\dots,x_k),x_k)$. 
		By \cref{rem:Zariski}, $\bs{\rho}(\Sigma)$ is a Zariski dense Anosov subgroup of $\SO^\circ(n,1)^k$. Then we have $\limitcone_{\bs{\rho}(\Sigma)} \subset 
		\inte \fa^+ \cup \{0\}=(0,\infty]^k \cup \{0\}$. Since 
		$\ker \bp\subset \{x_k=0\}$, we have 
		$\ker\bs{\varphi} \cap \limitcone_{\bs{\rho}(\Sigma)} = \{0\}$ and hence $\bs{\varphi}|_{\limitcone_{\bs{\rho}(\Sigma)}}$ is proper . 
		By \cref{lem:ProjectedCone}, $\bs{\varphi}(\limitcone_{\bs{\rho}(\Sigma)})$ is a convex cone. By the hypothesis on $\bs{\phi}$, the cone $\bs{\varphi}(\limitcone_{\bs{\rho}(\Sigma)})$ contains vectors $(y_1,1)$ and $(y_2,1)$ with $y_1< 0 < y_2$. Hence, $\bs{r} = (0,1) \in \interior \bs{\varphi}(\limitcone_{\bs{\rho}(\Sigma)})$.
		\cref{thm:Hyperbolic} follows from applying \cref{thm:Counting} to $\Gamma =\bs{\rho}(\Sigma)$, $\bs{\varphi}$ and $\bs{r}$. 
	\end{proof}

	\subsection*{Hitchin representations}
	Even in the case of a single Hitchin representation, \cref{thm:Counting} gives a new counting result.
	
	\begin{theorem}
		\label{thm:Hitchin1}
		Let $\rho:\Sigma \to \PSL_n\R$ be a Zariski dense Hitchin representation and $\bs{\phi}:\R^n \to \R$ be a linear form which is positive on $\interior\LieA^+$. Suppose that for some $1\le j\le n$, $\lambda_j(\rho(\Sigma))$ contains both positive and negative values; here $\lambda_j$ is as in \eqref{1}. Then there exists $\delta > 0$ such that for any $\e > 0$, we have as $T \to \infty$,
		$$\{[\sigma] \in [\Sigma]: 0 \le \lambda_j(\rho(\sigma)) \le \e, \, T \le  \bs{\phi}(\lambda(\rho(\sigma))) \le T + \e\} \sim c \cdot \frac{e^{\delta T}}{T^{3/2}};$$
		$$\{\sigma \in \Sigma: 0 \le \mu_j(\rho(\sigma)) \le \e, \, T \le  \bs{\phi}(\mu(\rho(\sigma))) \le T + \e\} \sim c \cdot c' \cdot \frac{e^{\delta T}}{T^{3/2}}$$
		for some constants $c, c' > 0$; here $\mu_j$ is as in \eqref{2}.
	\end{theorem}
	
	\begin{proof}
		Let $\bs{\varphi}: \R^n \to \R^2$ be given by $\bs{\varphi}(x_1,\dots,x_d) = (\bs{\phi}(x_1,\dots,x_d),x_j)$. Note that $\bs{\varphi}|_{\limitcone_{\rho(\Sigma)}}$ is proper since $\bs{\phi}$ is positive on $\interior\LieA^+$.  Then \cref{thm:Hitchin1} follows from applying \cref{thm:Counting} to $\Gamma =\rho(\Sigma)$, $\bs{\varphi}$ and the direction $\bs{r} = (1,0)$ which is an interior vector of the $\bs{\varphi}$-projection of $\limitcone_{\rho(\Sigma)}$ by the hypothesis on $\lambda_j(\rho(\Sigma))$. 
	\end{proof}

	\noindent \textbf{Proof of \cref{thm:Hitchin}}. Let $\bs{\rho}$ and $\bs{\beta}$ be as in \cref{thm:Hitchin}. A fact is that the only automorphisms of $\PSL_n\R$ are the inner automorphisms and their composition with the contragradient involution. Then by \cref{rem:Zariski}, $\bs{\rho}(\Sigma)$ is Zariski dense in $(\PSL_n\R)^d$.
	\cref{thm:Hitchin} follows from applying \cref{thm:Correlations} with $\bs{\varphi} = \bs{\beta}$ and the bounds use the fact that $\delta_{\rho_i,\beta_i} = 1$ for all $i$ \cite[Theorem B]{PS17}.

\end{document}